\documentclass[journal]{IEEEtran}
\usepackage{enumerate}
\usepackage{amsmath,amssymb,amsthm}
\usepackage{framed}
\usepackage{graphicx}
\usepackage{tikz}
\usepackage{epstopdf}
\usepackage{caption}
\usepackage{subcaption}
\usepackage{empheq}
\usepackage{hyperref}
\usepackage{cite}
\usepackage[ruled,vlined]{algorithm2e} 
\newcommand{\R}{\mathbb{R}}
\newcommand{\N}{\mathbb{N}}

\usepackage{color,soul}
\definecolor{lightblue}{rgb}{.90,.95,1}
\sethlcolor{lightblue}
\usetikzlibrary{arrows}

\newcommand{\tran}{^{\mbox{\scriptsize  \bf{T}}}}  
\newcommand{\bec}[1]{\bar{\vec{#1}}}

\newcommand{\dist}{\texttt{dist}}
\newcommand{\sign}{\texttt{sign}}
\newcommand{\Span}{\texttt{Span}}
\newcommand{\ang}{\texttt{ang}}

\newtheorem{defin}{Definition}
\newtheorem{theorem}{Theorem}

\newtheorem{lemma}{Lemma}
\newtheorem{corollary}{Corollary}
\newtheorem{remark}{Remark}

\newtheorem{example}{Example}

\newcommand{\be}{\begin{equation}}
\newcommand{\ee}{\end{equation}}
\newcommand{\ba}{\begin{array}}
\newcommand{\ea}{\end{array}}
\newcommand{\bea}{\begin{eqnarray}}
\newcommand{\eea}{\end{eqnarray}}

\newcommand{\vbar}{\raisebox{.17ex}{\rule{.04em}{1.35ex}}}
\newcommand{\vbarind}{\raisebox{.01ex}{\rule{.04em}{1.1ex}}}
\newcommand{\D}{\ifmmode {\rm I}\hspace{-.2em}{\rm D} \else ${\rm I}\hspace{-.2em}{\rm D}$ \fi}
\newcommand{\T}{\ifmmode {\rm I}\hspace{-.2em}{\rm T} \else ${\rm I}\hspace{-.2em}{\rm T}$ \fi}
\newcommand{\B}{\ifmmode {\rm I}\hspace{-.2em}{\rm B} \else \mbox{${\rm I}\hspace{-.2em}{\rm B}$} \fi}
\newcommand{\Hil}{\ifmmode {\rm I}\hspace{-.2em}{\rm H} \else \mbox{${\rm I}\hspace{-.2em}{\rm H}$} \fi}
\newcommand{\C}{\ifmmode \hspace{.2em}\vbar\hspace{-.31em}{\rm C} \else \mbox{$\hspace{.2em}\vbar\hspace{-.31em}{\rm C}$} \fi}
\newcommand{\Cind}{\ifmmode \hspace{.2em}\vbarind\hspace{-.25em}{\rm C} \else \mbox{$\hspace{.2em}\vbarind\hspace{-.25em}{\rm C}$} \fi}
\newcommand{\Q}{\ifmmode \hspace{.2em}\vbar\hspace{-.31em}{\rm Q} \else \mbox{$\hspace{.2em}\vbar\hspace{-.31em}{\rm Q}$} \fi}
\newcommand{\Z}{\ifmmode {\rm Z}\hspace{-.28em}{\rm Z} \else ${\rm Z}\hspace{-.38em}{\rm Z}$ \fi}

\renewcommand{\vec}[1]{{\bf{#1}}}     

\title{\LARGE \bf
Convergence of Limited Communications Gradient Methods 
}

\author{Sindri Magn\'{u}sson,   Chinwendu Enyioha, Na Li, Carlo Fischione, and Vahid Tarokh
 \thanks{
This work was supported by the VR Chromos Project, NSF 1548204, 1608509, and Career 1553407, and ARPA-E under the NODES program.
}
\thanks{S. Magn\'{u}sson and C.~Fischione are with the Electrical Engineering School, KTH Royal Institute of Technology, Stockholm, Sweden. (e-mail: sindrim@kth.se; carlofi@kth.se) }%
\thanks{C. Enyioha, N. Li, and V. Tarokh are with the School of Engineering and Applied Sciences, Harvard University, Cambridge, MA~USA. E-mail:{ sindrim@kth.se}. (email: cenyioha@seas.harvard.edu; nali@seas.harvard.edu; vahid@seas.harvard.edu)
}%
}

\begin{document}

\maketitle

\begin{abstract}
Distributed optimization increasingly plays a central role in economical and sustainable operation of cyber-physical systems.  Nevertheless, the complete potential of the technology has not yet been fully exploited in practice due to communication limitations posed by the real-world infrastructures.  This work investigates fundamental properties of distributed optimization based on  gradient methods, where gradient information is communicated using limited number of bits.  In particular, a general class of quantized gradient methods are studied where the gradient direction is approximated by a finite quantization set.  Sufficient and necessary conditions are provided on such a  quantization set to guarantee that the methods  minimize any  convex objective function with Lipschitz continuous gradient and a nonempty and bounded set of optimizers.  A lower bound on the cardinality of the quantization set is provided, along with specific examples of minimal quantizations.  Convergence rate results are established that connect the fineness of the quantization and the number of iterations needed to reach a predefined solution accuracy.  Generalizations of the results to a relevant class of  constrained problems using projections are considered.  Finally, the results are illustrated by simulations of practical systems.
\end{abstract}

 \section{Introduction}

  Recent advances in  distributed optimization have enabled more economical and sustainable control and operation  of cyber-physical systems. 
   However, these systems usually assume the availability of high performing communication infrastructures, which is often not practically possible.    
   For example, although large scale cyber-physical systems such as power networks are equipped with a natural communication infrastructure given by  the power lines~\cite{Galli2011grid}, such a communication network has a limited bandwidth. 
    Instead,  research efforts in distributed operation of power networks usually assume high data rates and low latency 
      communication technologies that are, unfortunately, not affordable or available today. 
            Another example is given by wireless sensor networks~\cite{Fischione_2011}, where efficient usage of communication plays a central role. 
 In fact, these networks  are powered by battery sources for communication over wireless links; hence, they are constrained in how much transmission they engage.  
These communication limitations are especially harsh in underwater networks, where acoustic channels are generally used, which have strong bandwidth limits~ \cite{partan2007survey}.
    Light communications  are also essential  in coordinating data networks~\cite{bertsekas1992data},
  where the control channels that support the data channels are obviously allocated limited bandwidth.  
 Another relevant example is within the emerging technology of extremely low latency networking or tactile internet~\cite{Duris_2016}, where information, especially for real time control applications, will be transmitted with very low latencies over wireless and wired networks.  
 However, this comes at a cost of using short packets containing limited information.

In all the cyber-physical systems mentioned above, distributed optimization plays a central role. These systems are networks of nodes whose operations have to be optimized by local decisions yet the coordination information can only go through constrained communication channels. In this paper, we restrict ourselves to one of the most prominent distributed optimization methods, decomposition based on the gradient method, and we investigate  the fundamental properties of such a method in terms of coordination limitations and optimality.

\subsection{Related Literature}

  Decomposition methods in optimization have been widely investigated  in wired/wireless communication~\cite{Kelly_1998,low_1999,Palomar_2006,Palomar_2007},   power networks~\cite{li2011optimal,chen2010two}, and wireless sensor networks~\cite{Madan_2006}, among others.
   These methods are typically based on communicating gradient information from a set of source nodes to {users,} which then solve a simple local subproblem. 
    The procedure can be performed using i) one-way communication where the source nodes estimate the gradient using available information~\cite{low_1999,Zhao_2014,DallAnese_2016} or ii) two-way communication where users and sources coordinate to evaluate the gradient.
   We investigate the performance of such methods using one-way communication where the number of bits per coordination step are limited.

 Bandwidth constrained optimization has  already received attention in the literature~\cite{Lue_1987,Rabbat_2005,Nedic_2008_cdc,pu2015quantization,yi2014quantized,Herdtner_2000}. 
 Initial studies are found in~\cite{Lue_1987}, where Tsitsiklis and Luo provide  lower bounds on the number of bits that two  processors need to communicate to (approximately) minimize the sum of two convex functions each of which is only accessed by one processor. 
 More recently, the authors of~\cite{Rabbat_2005} consider a variant of incremental gradient methods~\cite{Nedic_2001} over networks where each node projects its iterate to a grid before sending the iterate to the next node. 
 Similar quantization ideas are considered~\cite{Nedic_2008_cdc,pu2015quantization,yi2014quantized} in the context of {consensus-type} subgradient methods~\cite{Nedic_2009}.
 The work in~\cite{Herdtner_2000} studies the convergence of standard interference function methods for power control in cellular wireless systems where base stations send binary signals to the users optimizing the transmit radio power. 
 Those papers consider only the original primal optimization problem,  without introducing its dual problem, where  quantized primal variables are communicated.   
 However,  many network and resource sharing/allocation optimization problems are naturally decomposed using duality theory, where it is  the dual \emph{gradients} that are communicated.  
 This motivates our studies of limited communication gradient methods.

      \subsection{Statement of Contributions}

      The main contribution of this paper is to investigate the convergence of gradient methods where gradients are communicated using limited bandwidth.  
   We first   consider gradient methods where each coordinate of the gradient is communicated using only one-bit per iteration.     
   This setup is motivated by dual decomposition applications where a single entity maintains each dual variable, e.g., in the TCP control in~\cite{low_1999} each dual variable is maintained by one flow line.  
 Since dual problems are either unconstrained or constrained to the positive orthant $\R_+^N$,\footnote{Depending on whether the primal problem has inequality or equality constraints.} we consider both unconstrained problems and problems constrained to $\R_+^N$.   
 We prove that when the step size $\gamma >0 $ of the gradient method is fixed, then the iterates converge  approximately to the set of optimal solutions within some $\epsilon>0$ accuracy in a finite number of iterations, where $\epsilon$ tends to 0 as $\gamma$ converges to 0. 
 Moreover, we provide an upper complexity bound on the number of iterations needed to reach any $\epsilon>0$ accuracy. 
 This upper bound grows proportionally to $1/\epsilon^2$ as $\epsilon$ goes to zero for unconstrained problems,  and proportionally to $1/\epsilon^4$ for problems constrained by $\R_+^N$.  
 We also prove that if the step-sizes $\gamma(t)$ are non-summable and converge to 0, then the iterates 
  converge to the set of optimal solutions.

 The second contribution of the paper is to investigate the convergence of more general class of quantized gradient methods (QGM), where the gradient direction is quantized at every iteration.  
 We start by identifying  necessary  and sufficient conditions on the quantization 
 so that  the QGMs can minimize any convex objective function with Lipschitz continuous gradients and a nonempty and bounded set of optimizers. 
 We show that the minimal quantizations that satisfy these conditions have the cardinality $N+1$, where $N$ is the problem dimension. 
  We prove that for fixed  step-sizes $\gamma>0$ the iterates converge  approximately to the set of optimal solutions within some $\epsilon$-accuracy, where $\epsilon$ converges to 0 as $\gamma$ converges to 0.
   We provide an upper complexity bound on the number of iterations $T$ needed to reach any solution accuracy $\epsilon>0$. This upper bound depends on the fineness of the quantization. 
    Moreover, we show that the solution accuracy $\epsilon>0$ converges to zero at a rate proportional to $1/\sqrt{T}$ or $1/\sqrt{b}$ where $T$ and $b$ are the numbers of iterations and communicated bits, respectively. 
  We show that when the step-sizes $\gamma(t)$ are non-summable and converge to zero then the iterates asymptotically converge to the set of optimal solutions.

  A conference version of this work including parts of Section~\ref{sec:QGDM} and Section~\ref{sec:Convergence} appeared in~\cite{Magnusson_2016}, but without most of the proofs. 
   The rest of the work is presented here for the first time.
   Our previous papers~\cite{Enyioha_2016,Magnusson_2016_2} consider similar resource allocation problems as in this paper without communication constraints.

    \subsection{Notation}
     Vectors and matrices are represented by boldface lower and upper case letters, respectively.
   The set of real, positive real, and natural numbers, are denoted by $\R$, $\R_+$, and $\N$, respectively. 
 The set of real and positive $n$ vectors and $n{\times} m$ matrices are denoted by $\R^n$, $\R_+^n$, and $\R^{n\times m}$, respectively.
  Other sets are represented by calligraphy letters.
  We denote by $\mathcal{S}^{N-1}$, $\mathcal{S}^{N-1}(\vec{x},R)$, and $\mathcal{B}^N(\vec{x},R)$, respectively, the unit sphere in $\R^N$ and the sphere and open ball centred at $\vec{x}$ with radius $R$ in $\R^N$. 
 The superscript $(\cdot)\tran$ stands for transpose. 
 We let $[\vec{x}]_{\mathcal{X}}$ and $\lceil\vec{x}\rceil_+$ denote the projection of $\vec{x}$ to the sets $\mathcal{X}$ and $\R_+$.  
  $||\cdot||$ denotes the $2$-norm. $\nabla f$ is the gradient of $f$. 
 The distance between a vector $\vec{x}\in\R^N$ and a  set $\mathcal{X}\subseteq \R^N$ is denoted by $\dist(\vec{x},\mathcal{X})=\inf_{\vec{v}\in \mathcal{X}} ||\vec{v}-\vec{x}||$.

 \section{Preliminaries and Motivating Examples } \label{sec:Related Background}

  In this paper we consider optimization problems of the form
 \begin{equation}  \label{main_problem}
\begin{aligned} 
    & \underset{\vec{x}\in \R^N}{\text{minimize}}  
    & & f(\vec{x}), \\
    & \text{subject to} 
    & & \vec{x}\in \mathcal{X},
\end{aligned}
\end{equation}
 where $f{:}\R^N\rightarrow \R$. 
 We denote by $f^{\star}$ and $\mathcal{X}^{\star}\subseteq \mathcal{X}$ the optimal value and the set of optimizers to  Problem~\eqref{main_problem}, respectively. 
 We consider the following class of optimization problems: 
 \begin{defin} \label{assump:problem_feasible}
 Let $\mathcal{F}_L(\mathcal{X})$ denote the set of optimization problems of the form of Equation~\eqref{main_problem} where the function $f$ is convex and differentiable  with $L$-Lipschitz continuous gradient, $\mathcal{X}$ is closed and convex set, and the optimal  solution set $\mathcal{X}^{\star}$ is nonempty and bounded. 
  We write $f\in \mathcal{F}_L(\mathcal{X})$ to indicate that the optimization Problem~\eqref{main_problem} with the objective function $f$ in the class $\mathcal{F}_L(\mathcal{X})$.
 \end{defin}

\noindent For $f\in\mathcal{F}_L(\mathcal{X})$, it is well known that the gradient method
  \begin{align} \label{gradient_alg}
     \vec{x}(t{+}1) = \left[ \vec{x}(t)- \gamma(t) \nabla f(\vec{x}(t))\right]_{\mathcal{X}},
  \end{align}
  converges to $\mathcal{X}^{\star}$ under appropriate step-size rules~\cite{nonlinear_bertsekas}.
 When only the gradient \emph{direction} is known, the above iterates become 
  \begin{align} \label{gradient_dir_alg}
     \vec{x}(t{+}1) = \left[\vec{x}(t)- \gamma(t)\frac{ \nabla f(\vec{x}(t))}{||\nabla f(\vec{x}(t))||}\right]_{\mathcal{X}},
  \end{align}
  where we set $ \vec{x}(t{+}1) =\vec{x}(t)$ if $||\nabla f(\vec{x}(t))||=0$. 
For appropriate diminishing step-size rules, the iterates   converge to $\mathcal{X}^{\star}$ and for fixed step-size,  the stopping condition $f(\vec{x}(t))-f^{\star}<\epsilon$ can be achieved for any $\epsilon>0$~\cite{Book_Shor_1985}.

 Problems on the form of Equation~\eqref{main_problem} often appear as dual problems used to decompose optimization problems with coupling constraints~\cite{Palomar_2006,Palomar_2007}.  
 There, a distributed solution approach is achieved by solving the dual problem using gradient methods as given in Equations~\eqref{gradient_alg} and~\eqref{gradient_dir_alg}.  
 The dual gradient $\nabla(f(\vec{x}(t))$ can often be measured over the course of the algorithm, as it is the constraint violation in the primal problem, given dual variable $\vec{x}(t)$ (see, e.g.,~\cite{low_1999,Zhao_2014,DallAnese_2016}).  
 To perform the gradient update in Equations~\eqref{gradient_alg} and~\eqref{gradient_dir_alg} the gradient or gradient direction must be communicated, as illustrated in the following examples. 
 However, since communication resources are scarce  in many networks,  we consider another variant of the gradient method in~\eqref{gradient_alg}.
 That is
  \begin{align*} 
     \vec{x}(t{+}1) = \left[ \vec{x}(t)- \gamma(t)\vec{d}(t)\right]_{\mathcal{X}},
  \end{align*}  
  where $\vec{d}(t)$ is a quantized gradient direction coded using limited number of bits. 
 Before introducing the details of our quantized methods we  provide some application examples.

 \subsection{TCP Flow Control}
  \label{subsec:Application_RA}

 Consider a communication network with $N$ undirected links and $S$ data sources.  
 Let  $\mathcal{L}$ and $\mathcal{S}$ denote the ordered sets $\{1,\ldots, N\}$ and $\{1,\ldots,S\}$. 
 Denote the capacity of  link $l\in \mathcal{L}$ by $c_l>0$ and the transmission rate of source $s\in \mathcal{S}$ by $q_s \in [m_s,M_s]$, where $m_s$ and $M_s$ are upper and lower bound on the source.  
  Source $s\in \mathcal{S}$ has utility function $U_s:[m_s,M_s]\longrightarrow \R$.
  The data from source $s\in\mathcal{S}$ flows through a path consisting of links $\mathcal{L}_s \subseteq \mathcal{L}$ to its destination. We denote by $\mathcal{S}_l\subseteq \mathcal{S}$ the sources that use link $l\in \mathcal{L}$, i.e.,  $\mathcal{S}_l:= \{s\in \mathcal{S} | l \in \mathcal{L}_s\}$.
 Then the TCP flow control is to find data rates $q_s$, $s\in \mathcal{S}$, that solve the following optimization problem~\cite{low_1999,Wei_2013,Wei_2013_2,Beck_2014}
\begin{equation}  \label{RA}  
\begin{aligned} 
    & \underset{ q_1,\ldots, q_S}{\text{maximize}} 
    & & \sum_{s=1}^S U_s( q_s)   \\ 
    & \text{subject to}       
   & & \sum_{s\in \mathcal{S}_l} q_s \leq c_l, & \text{ for } l=1,\ldots, N   \\
   & & &    q_s \in [m_s,M_s], &\text{ for } s=1,\ldots, S. 
\end{aligned} 
\end{equation}
  For notational ease, we write $\vec{q}=(q_s)_{s\in \mathcal{S}}$, $\mathcal{Q}= \prod_{s\in \mathcal{S}} [m_s,M_s]$, 
  $\vec{c}=(c_l)_{l\in \mathcal{L}}$, and $\vec{A}\in \R^{N \times S}$ where 
  \begin{align}  \label{eq:TCPcontrol_A_matrix}
      \vec{A}_{ls}=
   \begin{cases}
    1 & \text{ if } l\in \mathcal{S}_s, \\
    0 & \text{ otherwise. }
   \end{cases}      
  \end{align}
  The dual problem of~\eqref{RA} is of the form~\eqref{main_problem} where $\mathcal{X}=\R_+^N$
 and the dual function $f:\R^N\rightarrow \R$ is  given by
  \begin{align*}
      f(\vec{x})=& \underset{\vec{q}\in \mathcal{Q}}{\text{maximize}}~  L(\vec{q},\vec{x}) =  L(\vec{q}(\vec{x}),\vec{x}),
  \end{align*}
  where 
 \begin{align}
             L(\vec{q}, \vec{x}) &=   \sum_{i=1}^M U_i( q_i) -\vec{x}\tran ( \vec{A}\vec{q} -\vec{c}),  \notag \\
         q_i(\vec{x}) &= \underset{ q_i\in [m_i,M_i]}{\text{argmax}}~ U_i(q_i)-  q_i\sum_{l\in \mathcal{L}_s} x_l. \label{eq:userSubproblem}
 \end{align}
\noindent 
 The dual gradient is given by $\nabla f(\vec{x})=\vec{c}-\vec{A} \vec{q}(\vec{x})$. 
 The dual gradient is bounded since the set $\mathcal{Q}$ is compact.
 Moreover, the set of optimal dual variables is bounded from Lemma 1 in~\cite{nedic2009approximate} 
 and the dual gradient $\nabla f(\cdot)$ is $L$-Lipschitz continuous from Lemma 3 in~\cite{low_1999},  provided that $U_i(\cdot)$ are strongly concave.  
 Therefore, the dual iterates $\vec{x}(t)$ in Equation~\eqref{gradient_alg} or~\eqref{gradient_dir_alg} and the associated primal iterates $\vec{q}(t)=\vec{q}(\vec{x}(t))$ converge to the optimal primal/dual solution of the optimization Problem~\eqref{RA}, provided that $\gamma(t)$ are chosen properly.

Dual gradient methods are desirable because they entail distributed solution to Problem~\eqref{RA} since Subproblems~\eqref{eq:userSubproblem} can be solved without any   coordination between the sources $\mathcal{S}$.  
  Moreover, the gradient component $\nabla_l f(\vec{x})=c_l- \sum_{s\in \mathcal{L}(l)} q_s(x_l)$ can often be measured at the data link $l$ since it is simply the difference between the link capacity, $c_l$, and the data transferred through the link~\cite{low_1999}. 
Therefore, the algorithm can be accomplished using one-way communication where each iteration $t$ consists of the following steps: 
 (i) the links broadcast $\vec{x}(t)$ to the sources 
 (ii) the sources solve the local Subproblem~\eqref{eq:userSubproblem} and then transfer the source at the data rate $q_i(\vec{x}(t))$, and (iii) the links measure the dual gradient $\nabla_l f(\vec{x}(t))$, the data flow through the line, to make the update~\eqref{gradient_alg} or~\eqref{gradient_dir_alg}.

   The control channels used to coordinate communication networks are often bandwidth limited.
   Hence, it is not practically feasible to broadcast the real-valued vector $\nabla f(\vec{x}(t))$ to the users.
 The questions we address in this paper are these: can we still solve the optimization problem by communicating quantized versions of the gradient? And what are the trade-offs between optimality and quantization? 
   This motivates our investigation of limited communication gradient methods.

 \subsection{Optimal Network Flow} 
 \label{subsec:Application_RA-ONF} 

 Consider a directed network $(\mathcal{N},\mathcal{E})$ where $\mathcal{N}=\{1,\ldots,N\}$ and $\mathcal{E}=\{1,\ldots, E\}$ denote  the sets of nodes and edges, respectively. 
 Let $v_e$ denote the flow through the edge $e\in \mathcal{E}$.
 The flow through the network can then be expressed by the matrix $\vec{A}\in\R^{N\times E}$ defined as
 \begin{align*}
   \vec{A}_{ne} := \begin{cases}  
                        1 & \text{ if edge $e$ leaves node $n$,} \\
                       - 1 & \text{ if edge $e$ enters node $n$, } \\ 
                        0 & \text{ otherwise.}
                   \end{cases} 
 \end{align*}  
\noindent Component $n\in \mathcal{N}$ of $\vec{A}\vec{v}$ indicates the flow injection/consumption at node $n$, where $\vec{v}=(v_1,\ldots,v_E)$.
 We assume that the flow injection ($c_n>0$) or consumption ($c_n\leq 0$) of node $n\in \mathcal{N}$ is $c_n\in \R$ and set $\vec{c}=(c_1,\ldots,c_N)\in \R^N$.
 Then the Optimal Network Flow problem is~\cite{convex_boyd,Ghadimi_2013,Zargham_2014} 
\begin{equation*} 
\begin{aligned} 
    & \underset{ v_1,\ldots, v_E}{\text{maximize}} 
    & & \sum_{e\in \mathcal{E}} - C_e( v_e)   \\ 
    & \text{subject to}       
   & & \vec{A} \vec{v}= \vec{c}, 
\end{aligned} 
\end{equation*}
 where $C_e:\R\rightarrow \R$ are cost functions of the flow through edge $e\in \mathcal{E}$.
 If $C_e$ are $\mu$-strongly convex 
 then the dual gradient is $L$-Lipschitz continuous with $L=\lambda_{\max}(\vec{A}\vec{A}\tran)/\mu$, see Lemma~1 in~\cite{Ghadimi_2013}.
 Then similar one-way communication dual decomposition algorithm can be performed as in Section~\ref{subsec:Application_RA}.
 In contrast to Section~\ref{subsec:Application_RA}, the dual problem is unconstrained, i.e., $\mathcal{X}=\R^N$. 
 In addition, the dual variables are maintained by the nodes so it is the nodes that broadcast the dual gradients while the edges solve the local subproblems. 
 Nevertheless, the dual gradient $\nabla f(\vec{x})$ can be measured at the nodes as the component $\nabla_n f(\vec{x})$ is simply the flow injection/consumption of node $n$ for a given $\vec{x}$.

 \subsection{Task Allocation} \label{Sec:TaskAllocation}

  Consider the problem of continuous allocation of $N$ tasks between $K$ machines.
  The sets of tasks and machines are denoted by $\mathcal{N}=\{1,\ldots, N\}$ and $\mathcal{K}=\{1\ldots,K\}$, respectively. 
   Let $\vec{c}\in \R^N$ denote the total amount of each task that needs to be completed.  
  The amount of each task done by machine $k\in\mathcal{K}$  is represented by the vector $\vec{w}_k\in \mathcal{W}_k \subseteq \R^N$, where $\mathcal{W}_k$ is a local  constraint of machine $k$.
  Then the goal is to find the task allocation that minimizes the cost: 
   \begin{equation} \label{RA-TASKs} 
\begin{aligned} 
    & \underset{ \vec{w}_1,\ldots,\vec{w}_K }{\text{maximize}} 
    & & \sum_{k\in \mathcal{K}} - C_k(\vec{w}_k)  \\ 
    & \text{subject to}       
   & & \sum_{k\in\mathcal{K}} \vec{w}_k = \vec{c},  \\
   &  & &    \vec{w}_k \in \mathcal{W}_k, ~~\text{for } k\in \mathcal{K},
\end{aligned} 
\end{equation}
where $C_k:\R^N\longrightarrow \R$ is the cost of performing the different tasks on machine $k\in \mathcal{K}$. 
  If $C_k$ are $\mu$-strongly convex then the dual gradient is $L$-Lipschitz continuous with $L=(K+1)/\mu$, see Lemma~1 in~\cite{Beck_2014}.
 Therefore, dual gradient methods~\eqref{gradient_alg} and~\eqref{gradient_dir_alg}  can solve the problem. 
 If the task manager can measure the total amount done of each task, i.e., the dual gradient, then a similar one-way communication coordination scheme as in Sections~\ref{subsec:Application_RA} and~\ref{subsec:Application_RA-ONF} can solve Problem~\eqref{RA-TASKs}.

 As shown later, the Lipschitz constant $L$ will be used to characterize several complexity bounds. Since $L$ on the dual gradient of the examples above is a function of  primal problem parameters such as the number of users and the concavity parameter $\mu$, those parameters affect the complexity bounds as well.

 \section{Quantized Gradient Descent Methods} \label{sec:QGDM}

 We consider general quantized gradient methods of the form
  \begin{equation}  \label{gradient_dir_quant}
\begin{split}  
     \vec{x}(t{+}1) =&~ \left[ \vec{x}(t)- \gamma(t)\vec{d}(t)\right]_{\mathcal{X}}, 
   \end{split}  
  \end{equation}
   where $\vec{d}(t)\in \mathcal{D}\subseteq \mathcal{S}^{N-1}$ is a finite set of quantized gradient directions. 
  The following relation is between the cardinality of $\mathcal{D}$ and communicated bits of each Iteration~\eqref{gradient_dir_quant}.
  \begin{remark} \label{remark:nr_of_bits}
      The set $\mathcal{D}$ can be coded using $\log_2(|\mathcal{D}|)$ bits.
  \end{remark}
 \noindent We now introduce the quantization sets $\mathcal{D}$  used in this paper. 
  
  \subsection{Binary Quantization} \label{sec:BQ}
   In the application examples in Sections~\ref{subsec:Application_RA},~\ref{subsec:Application_RA-ONF}, and~\ref{Sec:TaskAllocation}, each dual variable is associated to a single problem component, i.e., a  link, user, or task, respectively. 
   For example, in the TCP control example in Section~\ref{subsec:Application_RA}, the dual variable  $\vec{x}_n$ is associated with link $n$. 
 Therefore, to achieve the dual gradient Algorithm~\eqref{gradient_alg} each link $l\in \mathcal{L}$ can measure its flow, i.e.,
the dual gradient component $\nabla_l f(\vec{x}(t))$, and then broadcast $\nabla_l f(\vec{x}(t))$ to the sources that use link $l$. 
 However, it might be infeasible to broadcast the full dual gradient when bandwidth is limited. 
 An alternative approach is to have the links broadcast  a binary signal indicating whether the associated dual variable is to be increased or decreased, i.e., link $l$ broadcasts $\sign(\nabla_l f(\vec{x}(t)))$.
 %
  Similarly, in the network flow problem in Section~\ref{subsec:Application_RA-ONF}, each node can measure the flow through the node and then broadcast a binary signal indicating the direction of the associated dual gradient component.    
 This quantization can be formally expressed as follows.
%
 \begin{example}[Signs of the gradients] \label{Example:grad_signs}  
 Consider the quantized gradient method in Equation~\eqref{gradient_dir_quant}. Set  
  \begin{align*} \mathcal{D}=\left\{ (1/\sqrt{N})(e_1,e_2,\ldots,e_N)  \big| e_i\in \{-1,1\} \right\}, 
  \end{align*}
  and take $\vec{d}(t)=\sign(\nabla f(\vec{x}(t)))$. 
   \end{example}

 By using this binary quantization, we  prove the convergence of the Iterates~\eqref{gradient_dir_quant} when $\mathcal{X}=\R^N$ and $\mathcal{X}=\R_+^N$.
 Therefore, our results cover both the case when the optimization Problem~\eqref{main_problem} is a dual problem associated with equality and inequality constrained primal problems.\footnote{When the primal problem have equality constraint then the dual problem is unconstraint, so $\mathcal{X}=\R^N$. Otherwise, if the primal problem has inequality constraints then $\mathcal{X}=\R_+^N$.}  
 Our results show that the Iterates~\eqref{gradient_dir_quant} using the quantization in Example~\ref{Example:grad_signs} converge (i) approximately to the set of optimal values when the step-sizes are fixed and (ii) asymptotically when the step-size are diminishing and non-summable.
 In section~\ref{sec:Conv_const}, we prove the convergence in the constrained case when $\mathcal{X}=\R_+^N$.
 The convergence in the unconstrained case $\mathcal{X}=\R^N$ is a special case of  the more general convergence results in Section~\ref{sec:Convergence}.

  \subsection{Fundamental Limit: Proper Quantization} 
    \label{subsec:FLPQ}
  When the quantization in Example~\ref{Example:grad_signs} is used in the TCP problem then there is no collaboration between the network links (or  the nodes in the Network flow problem).
 As a result $|\mathcal{D}|=2^N$ and $\log_2(2^N)=N$ bits are used to broadcast the quantized gradient direction per iteration.
 However, in many applications~\cite{Palomar_2006,Palomar_2007}, the dual problem is maintained by a single coordinator. 
 Therefore, an interesting question is whether even fewer than $N$ bits can be used per iteration when a single coordinator maintains the dual problem?
 In that case, what is the minimal quantization $|\mathcal{D}|$ so the Iterates~\eqref{gradient_dir_quant} can solve the optimization Problem~\eqref{main_problem}?
 More generally, for what quantization sets $\mathcal{D}$ do the Iterates~\eqref{gradient_dir_quant} converge to optimal solution to the Problem~\eqref{main_problem}? 
 To answer such questions we now formalize how  a quantization set $\mathcal{D}$ enables the Iterates~\eqref{gradient_dir_quant} to solve the optimization Problem~\eqref{main_problem}.   
 \begin{defin} \label{def:proper quantization} 
    Consider Iterations~\eqref{gradient_dir_quant}. 
   A finite set  $\mathcal{D}$ is a \emph{proper quantization} for the problem class $ \mathcal{F}_L(\mathcal{X})$ if for every    $f\in \mathcal{F}_L(\mathcal{X})$ 
   and every initialization $\vec{x}(0)\in \mathcal{X}$ we can choose $\vec{d}(t)\in \mathcal{D}$ and $\gamma(t)\in \R_+$, for all $t\in \N$,  such that 
      $\lim_{t\rightarrow \infty} \dist (\vec{x}(t),\mathcal{X}^{\star})= 0$.
 \end{defin}
\noindent Using  Definition~\ref{def:proper quantization} we investigate the following questions:
 \begin{enumerate}[A)]
  \item Are there equivalent  constructive conditions  that can be used to determine whether $\mathcal{D}$ is a \emph{proper quantization} or to construct such quantization sets? 
    \item What is the minimal quantization, i.e., size $|\mathcal{D}|$, for which $\mathcal{D}$ is a \emph{proper quantization}?
    \item What are the connections between the  fineness of the quantization, i.e., the size of $|\mathcal{D}|$,  and the possible convergence rate of the algorithm?
 \end{enumerate} 
 For the class $\mathcal{F}_L(\R^N)$ of unconstrained problems, the next few paragraphs answer all these questions. However, as shown in Section~\ref{subset:Sols-CC}, a proper quantization set $\mathcal{D}$ for the class $\mathcal{F}_L(\R^N)$ might not be a proper quantization for $\mathcal{F}_L(\mathcal{X})$ when $\mathcal{X}$ is a proper subset of $\R^N$.

\vspace{0.2cm}
 
\noindent \textbf{Question A):} Consider the following definition.
 \begin{defin} \label{defin:min_angle}
    The finite set $\mathcal{D}$ is a $\theta$-cover  
     if $\theta \in [0,\pi/2)$ and for all $\vec{g}\in \mathcal{S}^{N-1}$ there is $\vec{d}\in \mathcal{D}$ such that
      $\ang(\vec{g},\vec{d}) = \cos^{-1} (\langle \vec{g},\vec{d} \rangle)  \leq \theta$. 
   Equivalently, for all $\vec{g}\in \mathcal{S}^{N-1}$ there is $\vec{d}\in \mathcal{D}$ such that 
      $\cos(\ang(\vec{g},\vec{d})) \geq \cos(\theta)>0$. 
 \end{defin}
 \noindent 
Informally,  $\mathcal{D}$  is a $\theta$-cover if for any non-zero vector in $\R^N$, there exists some element in $\mathcal{D}$ such that the angle between them is smaller than or equal to $\theta$.
The following theorem shows that Definition~\ref{defin:min_angle} of $\theta$-cover is actually equivalent to Definition~\ref{def:proper quantization} of proper quantization for the problem class $\mathcal{F}_L(\R^N)$. 
 \begin{theorem} \label{prop:equivalenceBetweenPQandTC}   
     Consider a finite set $\mathcal{D}\subseteq \mathcal{S}^{N-1}$.  $\mathcal{D}$ is a proper quantization for the class $\mathcal{F}_L(\R^N)$  (Definition~\ref{def:proper quantization}) if and only if there exists $\theta\in [0,2\pi)$ such that $\mathcal{D}$ is a $\theta$-cover (Definition~\ref{defin:min_angle}).     
  \end{theorem} 
  \begin{IEEEproof}
   The proof is found in Appendix~\ref{App:prop:equivalenceBetweenPQandTC}.
  \end{IEEEproof}
 Unlike Definition~\ref{def:proper quantization}  of proper quantization, Definition~\ref{defin:min_angle} is \emph{constructive} in the sense that it can be used to determine if a set $\mathcal{D}$ is a proper quantization and to construct such sets. 
 For example, we can use Definition~\ref{defin:min_angle} and Theorem~\ref{prop:equivalenceBetweenPQandTC} to show that the quantization scheme in Example~\ref{Example:grad_signs} is a proper quantization for the problem class $\mathcal{F}_L(\R^N)$.
  \begin{lemma} \label{Lemma:signThetaCov}
     The quantization in Example~\ref{Example:grad_signs} is a $\theta$-cover with $\cos(\theta)=1/\sqrt{N}$.
  \end{lemma}
  The lemma follows from  the fact that for any $\vec{x}\in \mathcal{S}^{N-1}$, if we choose $\vec{d}=(1/\sqrt{N}) \sign(\vec{x})$ then
     \begin{align*}
        \cos(\ang(\vec{x},\vec{d})) &= \langle \vec{x},\vec{d} \rangle 
         = \frac{1}{\sqrt{N}} \sum_{i=1}^N \vec{x}_i \cdot \sign(\vec{x}_i)  \\
         &\geq \frac{1}{\sqrt{N}} \sum_{i=1}^N \vec{x}_i^2 
         = \frac{1}{\sqrt{N}}  ||\vec{x}||^2 =  \frac{1}{\sqrt{N}}. 
    \end{align*}
  Lemma~\ref{Lemma:signThetaCov} proves that the quantization in Example~\ref{Example:grad_signs} is a proper quantization. 
  We give some other interesting $\theta$-covers  now.
  \begin{example}[Minimal Proper Quantization: $|\mathcal{D}|=N+1$] \label{Examples:theta-cover-1}   
    Set 
  \begin{align}  \label{eq:D:Examples:theta-cover-1}
       \mathcal{D}=\{ \vec{e}_1,\ldots,\vec{e}_N, -\vec{1}/\sqrt{N}\},
   \end{align}
  where $\vec{e}_i$ is the $i$-th element of the normal basis and $\vec{1}=(1,\ldots,1)\in \R^N$. 
  Clearly, $|\mathcal{D}|=N+1$ so $\mathcal{D}$ can be coded using only $\log_2(N+1)$ bits.
 We show below in Theorem~\ref{Theorem:MinimalProperQuant} that this is a minimal proper quantization as there does not exist proper quantization $\mathcal{D}$ with $|\mathcal{D}|\leq N$. 
  We show in Lemma~\ref{Lemma:Example1-longproofinappendix} in the  Appendix that $\mathcal{D}$ is a $\theta$-cover with 
  \begin{align} \label{eq:inExample1_theta} 
         \cos(\theta)= \frac{1}{\sqrt{N^2+2\sqrt{N}(N-1)}}.
  \end{align}
 \end{example}
 \begin{example}[Example in $\R^2$: $|\mathcal{D}|=n$]  \label{Examples:theta-cover-2}
         For every $n\in\N$ set 
     \begin{align*}
        \mathcal{D}_n=\left\{ \left[ \begin{array}{c} \cos(2\pi k/ n) \\ \sin(2\pi k/ n)\end{array}   \right] \in \R^2 \bigg| k=0,1,\ldots,n-1 \right\}.
     \end{align*} 
     Clearly,  if $n\geq 3$, $\mathcal{D}_n$ is a $\theta$-cover with $\theta = \pi/n$. 
 \end{example}
 \begin{example}[Normal Basis: $|\mathcal{D}|=2N$]
   Let 
   $\mathcal{D}=\{ \vec{e}_1,-\vec{e}_1, \vec{e}_2,-\vec{e}_2, \ldots,\vec{e}_N,-\vec{e}_N\}$. 
  $|\mathcal{D}|=2N$ and $\mathcal{D}$ is a $\theta$-cover with $\cos(\theta)=1/\sqrt{N}$
  since for all $\vec{x}\in \mathcal{S}^{N-1}$, 
  if we  choose $\vec{d}=  \sign(\vec{x}_i) \vec{e}_i$ where $i=\text{argmax}_{j=1,\ldots,N} |\vec{x}_j|$
  then 
   $ \cos(\ang(\vec{x},\vec{d}))=\langle \vec{x},\vec{d} \rangle=  \vec{x}_i \cdot \sign(\vec{x}_i)  = |\vec{x}_i| \geq  1/\sqrt{N}. $
 \end{example}

 For constant $\theta{\in} (0,\pi/2)$, it can be of interest to find the $\theta$-cover $\mathcal{D}$ which has minimal cardinality $|\mathcal{D}|$.  
 This problem has been investigated in the coding theory literature,~ \cite{Wyner_1967,sole1991covering}.

     \vspace{0.2cm}
 
\noindent \textbf{Question B):} The minimal proper quantization for the problem class $\mathcal{F}_L(\R^N)$ is $|\mathcal{D}|=N+1$. We already have a proper quantization with $|\mathcal{D}|=N+1$, see  Example~\ref{Examples:theta-cover-1}. 
     The following result shows that 
      there does not exist a  quantization set $\mathcal{D}$ with fewer elements than $N+1$. 
\begin{theorem} \label{Theorem:MinimalProperQuant}
  Suppose that $\mathcal{D}\subseteq \mathcal{S}^{N-1}$ and $|\mathcal{D}|\leq N$.  Then $\mathcal{D}$ is not a proper quantization.
 \end{theorem}
 \begin{IEEEproof}
   The proof is in Appendix~\ref{APP:Theorem:MinimalProperQuant}.
 \end{IEEEproof}
 This result shows that the minimum data-rate needed for the algorithm to  converge is $\log_2(N+1)$ bits/iteration. 
 To the best of our knowledge, there are no similar results on minimal quantizations for distributed optimization methods in the existing literature.

\vspace{0.2cm}
 
\noindent \textbf{Question C):}
  In Section~\ref{sec:Convergence}, we study the convergence of Iterates~\eqref{gradient_dir_quant} for the problem class $\mathcal{F}_L(\R^N)$ when $\mathcal{D}$ is a $\theta$-cover. 
  When the step-size is constant, i.e., $\gamma(t)=\gamma$, then we show that any solution accuracy $||\nabla f(\vec{x})||\leq \epsilon$ and $f(\vec{x})-f^{\star}\leq \epsilon$ can be achieved, for $\epsilon>0$. 
  We also give an upper bound on the number of iterations/bits needed to achieve that accuracy. 
  An implication of the results is that after $T\in \N$ iterations the accuracy 
   $$ ||\nabla f(\vec{x})||\leq \frac{M}{\cos(\theta) \sqrt{T}},$$ 
   can be reached using appropriate constant step-size choice, where $M>0$ is some constant.
   Finally, we show how to choose the step-sizes so that every limit point of the algorithm is an optimizer of Problem~\eqref{main_problem}.

\subsection{Fundamental Limit: What If There Are Constraints?} 
\label{subset:Sols-CC} 

 We now show by simple examples why a $\theta$-cover $\mathcal{D}$ might not be a proper quantization for the problem class $\mathcal{F}_L(\mathcal{X})$ when the feasible set $\mathcal{X}$ is a proper subset of $\R^N$. 
 These examples are illustrated in Fig.~\ref{fig:QGM-ConCha-A} and Fig.~\ref{fig:QGM-ConCha-B}.  
  Both figures demonstrate  scenarios where a single step of Iteration~\eqref{gradient_dir_quant} is taken from $\vec{x}$. 
 In both figures $\mathcal{D}{=}\{\vec{d}_1,\vec{d}_2,\vec{d}_3\}$ is a $\theta$-cover with $\theta{=}\pi/3$. 
 The feasible region is depicted by gray color. 
  The curves depict the contours of the objective function $f$. 
  The dotted lines depict the angle $\theta$. 

Fig.~\ref{fig:QGM-ConCha-A} depicts a scenario where Iteration~\eqref{gradient_dir_quant} may have a non-optimal stationary point, even though $\mathcal{D}$ is a $\theta$-cover. 
The point $\vec{x}$ is a stationary point since $-\vec{d}_1$ is orthogonal to the constraint.
 However,  $\vec{x}$ is not an optimal solution of Problem~\eqref{main_problem}, since the gradient $\nabla f(\vec{x})$ is not orthogonal to the constraint. 
 This example shows that the equivalence established in Theorem~\ref{prop:equivalenceBetweenPQandTC} does not generalize to the constrained case. 
Fig.~\ref{fig:QGM-ConCha-B} shows that the Iterates~\eqref{gradient_dir_quant} can go in the opposite direction of the optimal solution. 
 The Iterate~\eqref{gradient_dir_quant} is a not a descent direction; hence, the objective function value is increasing. 
 The optimal solution of Problem~\eqref{main_problem} and the Iterate~\eqref{gradient_dir_quant} are denoted by $\vec{x}^{\star}$  and $[\vec{x}-\gamma \vec{d}_1]_{\mathcal{X}}$, respectively.

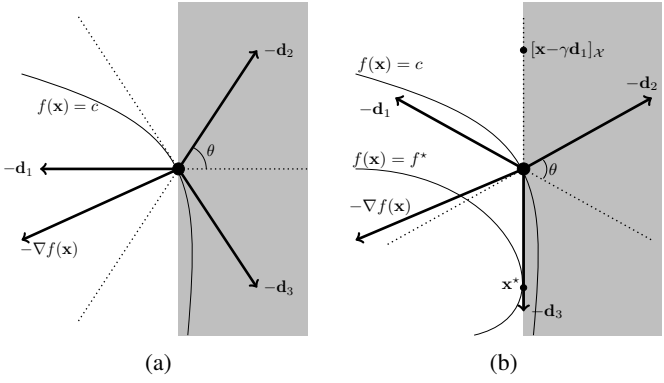
\begin{figure}[t]

\centering
    \begin{subfigure}[t]{0.5\columnwidth}
        \centering
        \resizebox{1\columnwidth}{\columnwidth}{
            \begin{tikzpicture}
                 \begin{scope}[thick,font=\scriptsize]
                 \end{scope}
                 \def\ringa{ circle (3)}

                 \shade[top color=lightgray, bottom color=lightgray] (0,-3.5) rectangle (2.8,3.5);
                 \begin{scope}[]
                      \draw [->,ultra thick] (0,0) -- (-3,0) node [ left] {$-\vec{d}_1$};
                      \draw [->,ultra thick] (0,0) -- (1.7,2.5) node [ right] {$-\vec{d}_2$};
                      \draw [->,ultra thick] (0,0) -- (1.7,-2.5) node [ right] {$-\vec{d}_3$};    
                      \draw [dotted,thick,color=black] (0,0) -- (-2.2,3.25) node [ right] {};         
                      \draw [dotted,thick,color=black] (0,0) -- (-2.2,-3.25) node [ right] {};    
                      \draw [dotted,thick,color=black] (0,0) -- (2.8,0) node [ right] {};    
                      \draw (0.6,0) arc (0:66:0.5);             
        
                      \draw [->,ultra thick,color=black] (0,0) -- (-3.4,-1.5);
                      \node [] at (-2.8,-1.7) {${-}\nabla f(\vec{x})$};                                                                                   
                 \end{scope}
                  \draw (-3.4,2) .. controls (0.3,0.9) and (0.5,-0.1) .. (0.2,-3.5);
                 \node[] at (0.7,0.4)  {$\theta$};
                 \node[] at (-2.4,1.3)  {$f(\vec{x})=c$};
                  \draw[black,fill=black] (0,0) circle (0.8ex);                

            \end{tikzpicture}
            }
           \caption{} \label{fig:QGM-ConCha-A}
           
    \end{subfigure}%
    ~
        \begin{subfigure}[t]{0.5\columnwidth}
        \centering
                \resizebox{1\columnwidth}{\columnwidth}{
            \begin{tikzpicture}
                 \begin{scope}[thick,font=\scriptsize]
                 \end{scope}
                 \def\ringa{ circle (3)}
                
                 \shade[top color=lightgray, bottom color=lightgray] (0,-3.5) rectangle (2.8,3.5);
                 \begin{scope}[]
                      \draw [->,ultra thick] (0,0) -- (150:3); 
                       \node [] at (-3,1.25) {${-}\vec{d}_1$};         
                      \draw [->,ultra thick] (0,0) -- (30:3); 
                       \node[] at (2.4,1.8)  {${-}\vec{d}_2$};
                      \draw [->,ultra thick] (0,0) -- (270:3)  node [ right] {${-}\vec{d}_3$};    
                      \draw [dotted,thick,color=black] (0,0) -- (90:3.2) node [ right] {};         
                      \draw [dotted,thick,color=black] (0,0) -- (210:3.2) node [ right] {};    
                      \draw [dotted,thick,color=black] (0,0) -- (330:3) node [ right] {};    
                      \draw (0.4,-0.2) arc (-40:40:0.3);          
        
                      \draw [->,ultra thick,color=black] (0,0) -- (-3.4,-1.5);          
                      \node [] at (-2.9,-0.8) {${-}\nabla f(\vec{x})$};                                                                         
                 \end{scope}
                  \draw (-3.4,2) .. controls (0.3,0.9) and (0.5,-0.1) .. (0.2,-3.5);
                  \draw (-3.4,0) .. controls (-0.3,0) and (1.05,-3) .. (-1,-3.5);
                 \node[] at (0.6,0)  {$\theta$};
                 \node[] at (-2.7,2.2)  {$f(\vec{x})=c$};
                 \node[] at (-2.7,0.2)  {$f(\vec{x})=f^{\star}$};
                 \node[] at (-0.25,-2.4)  {$\vec{x}^{\star}$};
                \draw[black,fill=black] (0,0) circle (0.8ex);
                  \draw[black,fill=black] (0,2.5) circle (0.4ex);
                  \draw[black,fill=black] (0,-2.5) circle (0.4ex);
                  \node[] at (0.9,2.5)  {$[ \vec{x}{-}\gamma \vec{d}_1]_{\mathcal{X}}$};
           
            \end{tikzpicture}
            }
        \caption{ } \label{fig:QGM-ConCha-B}
    \end{subfigure}
    \caption{
 An iteration of  Equation~\eqref{gradient_dir_quant} starting at $\vec{x}$. $\vec{x}$ is in the middle of the figure marked by a big filled circle. The feasible set $\mathcal{X}$ is marked by the shaded region.    $\vec{x}^{\star}$ and $f^{\star}$ are the optimizers and optimal value.  $\mathcal{D}=\{\vec{d}_1,\vec{d}_2,\vec{d}_3\}$ is a $\theta$-cover with $\theta=\pi/3$. In (a) $\vec{x}$ is a fixed point of the algorithm.  In (b) the algorithm takes the step $[ \vec{x}{-}\gamma \vec{d}_1]_{\mathcal{X}}$.} 
           \label{fig:QGM-ConCha}
\end{figure}

\section{Convergence - Binary Quantization} \label{sec:Conv_const}

 In this section, we  investigate the convergence of the quantized gradient method 
  \begin{align}
      \vec{x}(t+1) &= \left\lceil  \vec{x}(t)-  \frac{\gamma(t)}{\sqrt{N}} \sign(\nabla f(\vec{x}(t))) \right\rceil^+, \label{eq:prop:dim_step_size-Project-rec-1} 
  \end{align}  
  for solving the optimization Problem~\eqref{main_problem} when $\mathcal{X}=\R_+^N$. 
 We make the additional assumption
 that the gradients $\nabla f$ are $B$-bounded, i.e., $||\nabla f(\vec{x})||\leq B$ for all $\vec{x}\in \R_+^N$.\footnote{This assumption is only needed in this section. The dual gradient is generally bounded if the primal problem is strongly convex and has bounded feasible set, see Proposition~6.1.1 in~\cite{nonlinear_bertsekas}. 
 For example, the dual gradient is bounded in the TCP problem in Section~\ref{subsec:Application_RA}.} 
 In the next section, which considers unconstrained problems,  we allow the gradients to be unbounded. 
 
 In the analysis we take advantage of the following property of optimal solution $\vec{x}^{\star}$.
  \begin{lemma} \label{Lemma:OptCon}
     Consider optimization Problem~\eqref{main_problem} with $\mathcal{X}=\R_+^N$.
     For $\alpha>0$ define the function 
     \begin{align} \label{eq:closedFun} 
        L_{\alpha}(\vec{x})=|| \vec{x}-\lceil \vec{x}-\alpha \nabla f (\vec{x})\rceil^+||.
     \end{align}
     Then $\vec{x}\in \mathcal{X}^{\star}$ if and only if $L_{\alpha}(\vec{x})=0$.
  \end{lemma}
  \begin{IEEEproof}
     The  proof is found in Appendix~\ref{App:Lemma:OptCon}.
  \end{IEEEproof}
  We investigate the convergence of the iterates in Equation~\eqref{eq:prop:dim_step_size-Project-rec-1} when the step-sizes are fixed in Section~\ref{subseq:Pro-CSS} and when the step-sizes are diminishing in Section~\ref{subseq:Pro-DSS}.
 
 \subsection{Constant Step-Size} \label{subseq:Pro-CSS}

 In this section, we study the convergence of Iterates~\eqref{eq:prop:dim_step_size-Project-rec-1}  when the step-size $\gamma(t)$ is constant, i.e., when $\gamma(t)=\gamma$ for all $t$. 
 We show that the Iterates~\eqref{eq:prop:dim_step_size-Project-rec-1}  can approximately solve optimization Problem~\eqref{main_problem} up to any $\epsilon$-accuracy, provided that the step-size $\gamma>0$ is small enough. 
  By  approximately solving~\eqref{main_problem}, we mean that we can  find $\vec{x}\in \R_+^N$ that approximately satisfies certain optimality conditions. 
  In particular, we consider the following two types of optimality conditions for Problem~\eqref{main_problem}: 
 \begin{align}
   &  \mbox{Type-1:}   && L_{\alpha} (\vec{x}) \leq \epsilon,  \label{Type-1b} \\ 
   &  \mbox{Type-2:}  &&f(\vec{x}) -f^{\star} \leq \epsilon.   \label{Type-2b}  
\end{align} 
 A point $\vec{x}$ is an optimal solution to optimization Problem~\eqref{main_problem} if and only if~\eqref{Type-1b} [or~\eqref{Type-2b}] hold with $\epsilon=0$. 
 The \mbox{Type-1} optimality condition is a generalization of the optimality condition that $||\nabla f(\vec{x})||=0$ for unconstrained problems.
  The Type-2 optimality condition simply state that the distance from the obtained objective value to the optimal value is less than $\epsilon$.
We now show that both Type-1 and Type-2 approximate optimality conditions can be reached in finite number of iterations.

  \subsubsection{Stopping Condition of \emph{Type-1}}

  We start by showing the Type-1 optimality condition can be reached for all $\epsilon>0$ in finite number of iterations. 
  The following lemma is essential in  proving the result. 
\begin{lemma}  \label{lemma:binary_feedback_descent_lemma-Pojected}
 Suppose $f\in \mathcal{F}_L(\R_+^N)$ and $\nabla f$ is $B$-bounded.
 Suppose $\epsilon>0$ and $\alpha\geq 0$ are such that
 \begin{align}
    L_{\alpha}(\vec{x})= ||\vec{x}- \lceil \vec{x}- \alpha \nabla f(\vec{x}) \rceil^+ || \geq \epsilon. \label{eq:LemmaProDec-OC}
 \end{align}
Then the following holds
 \begin{align*}
   f\left( \left\lceil \vec{x}{-} \frac{\gamma}{\sqrt{N}} \sign(\nabla f(\vec{x}))  \right\rceil^+ \right) \leq f(\vec{x})-\bar{\delta}(\epsilon,\alpha,\gamma)
 \end{align*}
 where
  $$ \bar{\delta}(\epsilon,\alpha,\gamma) 
    = \left( \frac{2 \epsilon^2 }{L \alpha^2 B N^{3/2}}  - \gamma  \right) \frac{L}{2} \gamma.$$
 \end{lemma}
 \begin{IEEEproof}
  The proof is provided in Appendix~\ref{App:lemma:binary_feedback_descent_lemma-Pojected}.
 \end{IEEEproof} 
 The lemma shows that if $L_{\alpha}(\vec{x})>0$ for some  $\vec{x}\in \R_+^N$, then  objective function value $f(\vec{x})$ will decrease with   Iterates~\eqref{eq:prop:dim_step_size-Project-rec-1}, provided that the step-size $\gamma(t)>0$ is small enough. 
 We use this intuition to provide an upper bound on the number of iterations needed to achieve the Type-1 approximate optimality.  
  \begin{theorem} \label{Prop:FSS-SC-Type-1-Cons} 
   Suppose $f\in \mathcal{F}_L(\R_+^N)$, $\nabla f$ is $B$-bounded, $\vec{x}(t)$ are generated by Equation~\eqref{eq:prop:dim_step_size-Project-rec-1}, and define the set 
\begin{align}    \label{eq:inProP-FSS-SC-Type-1-MCXep-Proj}
      \bar{\mathcal{X}}_{\alpha}(\epsilon) := \{ \vec{x}\in  \R_+^N \big|  L_{\alpha}(\vec{x}) \leq \epsilon \}.
\end{align}   
 Then for any $\epsilon>0$ and $\gamma\in (0,2\epsilon^2/(\alpha^2 B N^{3/2}))$, there exists $T\in \N\cup \{0\}$ such that $\vec{x}(T)\in \bar{\mathcal{X}}_{\alpha}(\epsilon)$, with $T$ bounded by
                \begin{align}  \label{eq:inProp:-conStep-a-Drec-const}
                             T\leq \left\lceil \frac{2(f(\vec{x}(0))-f^{\star})\alpha^2 B N^{3/2}}{\gamma(2 \epsilon^2-L\gamma \alpha^2 B N^{3/2} )} \right\rceil.
                \end{align}
              The upper bound in Equation~\eqref{eq:inProp:-conStep-a-Drec-const} is minimized when the step-size 
                $\gamma^{\star}=\epsilon^2/(L \alpha^2 B N^{3/2})$ is used.         
  \end{theorem}
  \begin{IEEEproof}
Let $\epsilon>0$ be given and choose any  $\gamma\in (0,2\epsilon^2/(L \alpha^2BN^{3/2}))$. 
 Then from Lemma~\ref{lemma:binary_feedback_descent_lemma-Pojected} we have for all $\vec{x}(t)\notin \bar{\mathcal{X}}_{\alpha}(\epsilon)$ that
      \begin{align} \label{eq:inProp-conStep-a-Drec}
        f(\vec{x}(t+1)) \leq f(\vec{x}(t)) - \bar{\delta}(\epsilon,\alpha,\gamma),
     \end{align}
     where $\bar{\delta}(\epsilon,\alpha,\gamma)>0$. 
     By recursively applying Equation~\eqref{eq:inProp-conStep-a-Drec}, it follows that if $\vec{x}(t)\notin \bar{\mathcal{X}}_{\alpha}(\epsilon)$ for all $t<s$  
     then
      \begin{align} \label{eq:inProp-conStep-a-Drec-bbb}
      0\leq  f(\vec{x}(s)) -f^{\star} \leq f(\vec{x}(0)) -f^{\star} - s~ \bar{\delta}(\epsilon,\gamma, \theta).
     \end{align}
    Therefore, there must exist $T\leq \lceil(f(\vec{x}(0))-f^{\star})/ \bar{\delta}(\epsilon,\alpha, \gamma) \rceil$ such that $\vec{x}(T) \in \bar{\mathcal{X}}_{\alpha}(\epsilon)$; otherwise, we can use Equation~\eqref{eq:inProp-conStep-a-Drec-bbb} with $s=\lceil(f(\vec{x}(0))-f^{\star})/ \bar{\delta}(\epsilon,\alpha,\gamma) \rceil+1$ to get the contradiction that  $f(\vec{x}(s)) < f^{\star}$. 
  By rearranging $T\leq\lceil(f(\vec{x}(0))-f^{\star})/ \bar{\delta}(\epsilon,\alpha,\gamma) \rceil$, we obtain the bound in Equation~\eqref{eq:inProp:-conStep-a-Drec-const}.    The optimal step-size $\gamma^{\star}=\epsilon^2/(L \alpha^2 B N^{3/2})$ comes  by maximizing the denominator in Equation~\eqref{eq:inProp:-conStep-a-Drec-const}.   
  \end{IEEEproof}
  The theorem shows that the Type-1 approximate optimality condition can be reached in a finite number of iterations. 
  Moreover, Equation~\eqref{eq:inProp:-conStep-a-Drec-const} provides an upper complexity bound on the algorithm,  showing the number of iterations needed to reach any $\epsilon>0$ accuracy. 
  When  the step-size is $\gamma^{\star}=\epsilon^2/(L \alpha^2 B N^{3/2})$ then the upper bound in Equation~\eqref{eq:inProp:-conStep-a-Drec-const} increases proportionally to  $1/\epsilon^4$ as $\epsilon$ goes to zero. 
  In Section~\ref{sec:Convergence:constantStep}, we show that this bound can be improved for unconstrained problems (where it increases proportionally to  $1/\epsilon^2$ as $\epsilon$ goes to zero).

   \subsubsection{Stopping Condition of \emph{Type-2}}
     We now show that the Type-2 approximate optimality can be reached for any $\epsilon$ accuracy. 
     The following lemma is used to obtain the result. 
\begin{lemma}  \label{prop:fixed-stepsize-type-2-Proj}
   Suppose that $f\in \mathcal{F}_L(\R_+^N)$, $\nabla f$ is $B$-bounded, and the iterates $\vec{x}(t)$ are generated by Equation~\eqref{eq:prop:dim_step_size-Project-rec-1}. 
 Then for all $\epsilon>0$, $\alpha>0$, $\gamma(t) \in (0,\bar{\gamma})$, with $\bar{\gamma}=2\epsilon^2/(\alpha^2 B N^{3/2})$, 
  and $T\in \N$ such that $\vec{x}(T)\in \bar{\mathcal{X}}_{\alpha}(\epsilon)$   
       \begin{align} \label{inProp:SC-T2-a-Proj}    
           f(\vec{x}(t)) \leq \bar{F}_{\alpha}(\epsilon) + \frac{L}{2}  \bar{\gamma}^2, ~\text{ for all $t\geq T$,}
       \end{align}
       where  $\bar{F}_{\alpha}:\R_+\rightarrow \R\cup \{\infty\}$ is given by
        \begin{align}  
           \bar{F}_{\alpha}(\kappa) = \sup\{ f(\vec{x}) | \vec{x}\in \bar{\mathcal{X}}_{\alpha}(\kappa)\}. \label{inProp:SC-T2-a3-Proj}  
        \end{align}
        Moreover, there exists $\kappa>0$ such that for all $\epsilon\in [0,\kappa]$ following holds (i) $\bar{\mathcal{X}}_{\alpha}(\epsilon)$ is bounded and  (ii) $\bar{F}_{\alpha}(\epsilon) <\infty$. It also holds that
              $ \lim_{\epsilon \rightarrow 0^+} \bar{F}_{\alpha}(\epsilon) = f^{\star}$.
 \end{lemma}        
 \begin{IEEEproof}
  The proof is provided in Appendix~\ref{App:prop:fixed-stepsize-type-2-Proj}.
 \end{IEEEproof}   
   The lemma is useful in deriving Type-2 optimality conditions since it
   connects the results from Theorem~\ref{Prop:FSS-SC-Type-1-Cons} to the quantity $f(\vec{x}(t))-f^{\star}$ via the function $\bar{F}_{\alpha}(\kappa)$ defined in Equation~\eqref{inProp:SC-T2-a3-Proj}. 
    In particular, the lemma provides a bound on $f(\vec{x}(t))$ that depends on $\bar{F}_{\alpha}(\epsilon)$ and the step-size. 
   Therefore, since $\lim_{\epsilon\rightarrow 0^+}\bar{F}_{\alpha}(\epsilon)=f^{\star}$ we can enforce $f(\vec{x}(t))-f^{\star}$ to be arbitrarily small after some time $T$, i.e., for all $t\geq T$, by choosing a small enough step-size.
The idea is  now  formalized. 
      \begin{theorem} \label{Theo:UnCon-FS-main-Cons}
             Suppose that $f\in \mathcal{F}_L(\R_+^N)$, $\nabla f$ is $B$-bounded, and the iterates $\vec{x}(t)$ are generated by Equation~\eqref{eq:prop:dim_step_size-Project-rec-1}.
       Then  for any $\epsilon>0$
       there exists step-size $\gamma>0$ and $T\in \N$ such that $f(\vec{x}(t))-f^{\star}\leq\epsilon$ for all $t\geq T$.     
    \end{theorem}
    \begin{IEEEproof}
       The result follows directly from Lemma~\ref{prop:fixed-stepsize-type-2-Proj} and Theorem~\ref{Prop:FSS-SC-Type-1-Cons}.
    \end{IEEEproof}     
 Theorem~\ref{Theo:UnCon-FS-main-Cons} proves that the Type-2 optimality condition [Eq.~\eqref{Type-2b}] can be achieved for any $\epsilon>0$ in a finite number of iterations. 
  However, unlike for the Type-1 optimality condition, we could not provide an explicit step-size choice or a bound on the number of iterations needed to achieve the $\epsilon>0$ accuracy.

 \subsection{Diminishing Step-Size} \label{subseq:Pro-DSS}

We now consider the diminishing step-size case. 
The following result shows that the step-sizes can be chosen so that Iterates~\eqref{eq:prop:dim_step_size-Project-rec-1}  converge asymptotically to the optimal solutions of Problem~\eqref{main_problem}.

\begin{theorem} \label{prop:dim_step_size-Project}

    Suppose that $f\in \mathcal{F}_L(\R_+^N)$, $\nabla f$ is $B$-bounded, 
   and the iterates $\vec{x}(t)$ are generated by Equation~\eqref{eq:prop:dim_step_size-Project-rec-1} where
    $$\lim_{t\rightarrow \infty} \gamma(t)=0 ~\text{ and }~\sum_{t=0}^{\infty} \gamma(t)=\infty.$$
  Then $\vec{x}(t)$ converges to the set of optimal solutions of the optimization Problem~\eqref{main_problem}, i.e.,   $\lim_{t\rightarrow \infty} \dist(\vec{x}(t),\mathcal{X}^{\star})=0$.
\end{theorem}
\begin{IEEEproof}
 Proof is provided in Appendix~\ref{App:subseq:Pro-DSS}
\end{IEEEproof}
 The step-size choice in the theorem is also necessary to ensure that $\lim_{t\rightarrow \infty} \dist(\vec{x}(t),\mathcal{X}^{\star})=0$ holds for all $f\in \mathcal{F}_+(\R_+^N)$ with $\nabla f$ being $B$-bounded. 
    To see why, consider the scalar function
    \begin{align}  \label{eq:flowerbound}
         f(x)=\begin{cases} 0.5(x-1)^2 & \text{ if } |x-1|\leq1 \\
                                       |x-1|-0.5 & \text{ otherwise}.\end{cases}
    \end{align}
    Then $f$ has the unique optimizer $x^{\star}=1$ and satisfies the assumptions of Theorem~\ref{prop:dim_step_size-Project}.

    Let us first show that if $\lim_{t\rightarrow \infty } \gamma(t) \neq 0$ then $\lim_{t\rightarrow \infty} \dist(x(t),\mathcal{X}^{\star}) {\neq}0$. 
    If $\gamma(t)>0$ does not converge to zero
    then there exists $I\in (0,1)$ and a subsequence $t_k$ such that $\gamma(t_k)\geq I$ for infinitely many $t\in \N$. 
    Then either $|x(t_k+1)-x^{\star}|$ or $|x(t_k)-x^{\star}|$ must be larger than or equal to $I/2$ for all $k\in \N$ because if  $|x(t_k)-x^{\star}|\leq I/2$ then $|x(t_k+1)-x^{\star}|\geq I/2$.
  Therefore,   $|x(t)-x^{\star}|\geq I/2$ for infinitely many $t\in\N$ so $\limsup_{t\rightarrow \infty} \dist(x(t),\mathcal{X}^{\star})\geq I/2$.      
   Let us  next show that if $\sum \gamma(t)<\infty$ then $\lim_{t\rightarrow \infty} \dist(x(t),\mathcal{X}^{\star}) {\neq}0$. 
   Take $x(0)=  2+ \sum_{t=0}^{\infty} \gamma(t)$.
   Then $x(t)\geq \vec{x}(0)-\sum_{\tau=0}^t \gamma(t) \geq 2$ for all $t\in \N$, so $\liminf_{t\rightarrow \infty} \dist(x(t),\mathcal{X}^{\star})\geq 1$.

 The advantage  of using diminishing step-sizes, as in Theorem~\ref{prop:dim_step_size-Project},   
  is that the algorithm can asymptotically converge to the set of optimal values.  
 Moreover,  diminishing step-size rules can be implemented even if problem parameters such as the Lipschitz constant $L$ are unknown, unlike when the step-size is constant. 
 On the other hand, it is more complicated to characterize the convergence rate, similar to Equation~\eqref{eq:inProp:-conStep-a-Drec-const}, when diminishing step-sizes are used.

 \section{Convergence - General Quantization}  \label{sec:Convergence}

 In the previous section, we studied quantized gradient methods where a particular quantization based on the sign of the gradient was used for constrained optimization problems.
 As we discussed in Section~\ref{subsec:FLPQ}, for unconstrained problems, a more general class of quantizations called $\theta$-covers  (Definition~\ref{defin:min_angle}) ensures that the quantized gradient methods can minimize any $f\in\mathcal{F}_L(\R^N)$.  
 In this section we formally prove this, i.e., if the quantization is a $\theta$-cover then the quantized gradient methods converge (i) approximately to an optimal solution when the step-sizes are constant and (ii) asymptotically to an optimal solution when the step-sizes are non summable and converge to zero.  Moreover, we study how the quantization fineness, i.e., $\theta$,  affects the algorithm convergence. 
  We first consider the case when the step-sizes are fixed, i.e., $\gamma(t)=\gamma$, in subsection~\ref{sec:Convergence:constantStep}.
  Then in subsection~\ref{sec:Convergence:DimStep} we consider diminishing step-sizes.
  
  \subsection{Constant Step-Size} \label{sec:Convergence:constantStep}

   Similar to Section~\ref{subseq:Pro-CSS},
 we consider the following two types of approximate optimality conditions:  
   \begin{align}
   &  \mbox{Type-1:}   && ||\nabla f (\vec{x} ) || \leq \epsilon,  \label{Type-1} \\ 
   &  \mbox{Type-2:}  &&f(\vec{x}(t)) -f^{\star} \leq \epsilon.   \label{Type-2}  
    \end{align}

 \subsubsection{Stopping Condition of \emph{Type-1}}

  We start by showing that Type-1 approximate optimality can be achieved for any $\epsilon>0$ in a finite number of iterations.   Further, we provide a lower and upper bound on the number of iterations needed to achieve the $\epsilon$-accuarcy (that depends on $\theta$).     
 A key result used to obtain the result is  the following lemma.   \begin{lemma}  \label{lemma:binary_feedback_descent_lemma}
    Suppose $f\in \mathcal{F}_L(\R^N)$,       $\epsilon>0$,  
    $\theta \in [0,\pi/2)$,
    $\vec{x}\in \R^N$,  
    $||\nabla f(\vec{x})|| {>} \epsilon$, 
    and $\vec{d}\in \mathcal{S}^{N-1}$ where $\ang( \nabla f( \vec{x}),\vec{d} ) {\leq} \theta$.
   Then
    \begin{align}
       f(\vec{x}-\gamma \vec{d}) \leq f(\vec{x})- \delta(\epsilon,\gamma, \theta),
    \end{align}
    where 
    \begin{align} \label{eq:inLemmaBD-delta}
      \delta(\epsilon,\gamma, \theta)= \left(  \frac{2 \cos(\theta) \epsilon}{L}  -\gamma    \right)  \frac{L}{2}\gamma.     \end{align}
    Clearly, $\delta(\epsilon,\gamma, \theta)> 0$ when $\gamma \in (0,2 \cos(\theta) \epsilon/L)$.  
 \end{lemma}
 \begin{IEEEproof}
  The proof is provided in Appendix~\ref{App:lemma:binary_feedback_descent_lemma}.
 \end{IEEEproof} 
 The lemma shows that if $||\nabla f(\vec{x})||>0$ and  $\ang( \nabla f( \vec{x}),\vec{d} ) \leq \theta<\pi/2$  for some $\vec{x}\in\R^N$ then 
  the objective function value can be decreased by taking a step in the direction $-\vec{d}$, i.e., $f(\vec{x}-\gamma \vec{d})\leq f(\vec{x})$ for small enough $\gamma>0$. 
  Therefore, if $\mathcal{D}$ is a $\theta$-cover then we can always find $\vec{d}\in \mathcal{D}$ and a step-size $\gamma>0$ such that $f(\vec{x}-\gamma \vec{d})\leq f(\vec{x})$.
 We now use this intuition from Lemma~\ref{lemma:binary_feedback_descent_lemma} to provide the upper and lower bounds on the number of iterations that are needed to reach the Type-1 optimality condition.

  \begin{theorem} \label{Prop:FSS-SC-Type-1} 
     Suppose that $f\in \mathcal{F}_L(\R^N)$, $\mathcal{D}$ is a $\theta$-cover (Definition~\ref{defin:min_angle}), the iterates $\vec{x}(t)$ are generated by Equation~\eqref{gradient_dir_quant},  and define the set 
     \begin{align}  \label{eq:inProP-FSS-SC-Type-1-MCXep}
       \mathcal{X}(\epsilon)=\{ \vec{x}\in \R^N \big| || \nabla f(\vec{x})|| \leq \epsilon \}.
      \end{align}        
     Then the following holds:
    \begin{enumerate}[a)]
       \item For any $\epsilon>0$, if $\gamma\in (0,2\cos(\theta) \epsilon / L)$ then there exists $T\in \N\cup \{0\}$ such that $\vec{x}(T)\in \mathcal{X}(\epsilon)$, with $T$ bounded by
                \begin{align}  \label{eq:inProp:-conStep-a-Drec}
                             T\leq \left\lceil \frac{2(f(\vec{x}(0))-f^{\star})}{\gamma(2 \cos(\theta)\epsilon-L\gamma)} \right\rceil.
                \end{align}
              The upper bound in Equation~\eqref{eq:inProp:-conStep-a-Drec} is minimized with the optimal step-size $\gamma^{\star}=\cos(\theta) \epsilon/L$.

        \item Given a fixed step-size $\gamma>0$ and scalar $\kappa>0$, if we choose 
          \begin{align}  \label{inProp:epsiloneq}
              \epsilon(\kappa,\gamma)=\kappa+\gamma L/(2\cos \theta)
          \end{align}      
         then there exists $T\in \N\cup \{0\}$ such that $\vec{x}(T)\in \mathcal{X}( \epsilon(\kappa,\gamma))$, with $T$ bounded by
                \begin{align}  \label{eq:inProp:-conStep-b-Drec}
                           T\leq \left\lceil \frac{f(\vec{x}(0))-f^{\star}}{\cos(\theta) \gamma \kappa} \right\rceil.
                \end{align}
                
         \item (Lower Bound on $T$) For any step-size $\gamma>0$ and $\epsilon>0$ if $\vec{x}(T)\in \mathcal{X}(\epsilon)$ then 
          \begin{align}
                   \frac{||\nabla f(\vec{x}(0))||-\epsilon}{\gamma L} \leq T
          \end{align}

    \end{enumerate} 

  \end{theorem}
  \begin{IEEEproof}
 a) The proof follows the same arguments as the proof of Theorem~\ref{Prop:FSS-SC-Type-1-Cons} using Lemma~\ref{lemma:binary_feedback_descent_lemma} in place of Lemma~\ref{lemma:binary_feedback_descent_lemma-Pojected}.

  b) The result can be obtained by substituting $\epsilon(\kappa,\gamma)$ in Equation~\eqref{eq:inProp:-conStep-a-Drec}.

  c)    
   Using the fact that the gradient $\nabla f$ is $L$-Lipschitz continuous we have
   $$ ||\nabla f(\vec{x}(t))-\nabla f(\vec{x}(t+1))||\leq L ||\vec{x}(t)-\vec{x}(t+1)||\leq L \gamma.$$
   Therefore, using the triangle inequality, we have 
    \begin{align*}
        ||\nabla f(\vec{x}(t))||-L\gamma \leq || \nabla f(\vec{x}(t{+}1))||~~\text{ for all $t\in \N$}.
    \end{align*}
   Recursively applying the inequality gives 
       \begin{align*}
        ||\nabla f(\vec{x}(0))||-L\gamma t \leq || \nabla f(\vec{x}(t))||.
    \end{align*}
    Hence, 
    $||\nabla f(\vec{x}(t)) || \leq \epsilon$ can only hold when $t\geq (||\nabla f(\vec{x}(0)) || -\epsilon)/(L\gamma)$.
  \end{IEEEproof}
   Theorem~\ref{Prop:FSS-SC-Type-1}-a) proves that if $\mathcal{D}$ is a $\theta$-cover then the Type-1 optimality condition [Eq.~\eqref{Type-1}] can be achieved with $\epsilon$-accuracy in a finite number of iterations, for all $\epsilon>0$. 
  Moreover, the theorem gives an upper bound on the number of iterations needed to achieve such $\epsilon$-accuracy.
   This bound decreases as $\theta$ decreases, i.e., as the quantization becomes finer.      Even though the bound in Equation~\eqref{eq:inProp:-conStep-a-Drec} is on the number of iteration, since $\log_2(|\mathcal{D}|)$ bits are communicated per iteration, the results shows that in the worst case scenario
   $$\left\lceil\frac{2(f(\vec{x}(0))-f^{\star})}{\gamma(2 \cos(\theta)\epsilon-L\gamma)} \right\rceil \log_2(|\mathcal{D}|) \text{ bits}$$ 
   are needed to find $\vec{x}\in\R^N$ such that  $||\nabla f(\vec{x})||\leq \epsilon$.

  Theorem~\ref{Prop:FSS-SC-Type-1}-b) demonstrates what $\epsilon$-accuracy can be achieved for a given step-size. 
 The parameter $\kappa$ captures a trade-off between the $\epsilon$-accuracy  and the number of iterations needed to achieve that $\epsilon$-accuracy. 
   By optimizing over both $\gamma$ and $\kappa$ in Theorem~\ref{Prop:FSS-SC-Type-1}-b) we can find an optimal bound on the accuracy $\epsilon$ that can be guaranteed in $T$ iterations. 
   That is to find $\gamma$ and $\kappa$ that solve the following optimization problem.
        \begin{equation} \label{eq:inProof:Type-1-OP}
    \begin{aligned}
      & \underset{\kappa, \gamma}{\text{minimize}}
      & & \epsilon(\kappa,\gamma) = \kappa+ \frac{ L}{2 \cos(\theta)} \gamma \\
      & \text{subject to}
      & & \frac{f(\vec{x}(0))-f^{\star}}{ \cos(\theta ) \gamma \kappa} \leq T, \\
      &&& \gamma, \kappa > 0.
    \end{aligned}
  \end{equation}
  Formally, this bound is given as follows.
   \begin{corollary} \label{prop:const_step-size:optimalAcc-given_itnum}

      For any $T\in \N$ we have:
       
      i) The minimal bound $\epsilon(\kappa,\gamma)$ [Eq.~\eqref{inProp:epsiloneq}] achieved in $T$ iterations, i.e., the solution to Problem~\eqref{eq:inProof:Type-1-OP}, is
     \begin{align} \label{eq:inProof:Type-1-OP-sol1}
           \epsilon^{\star} = \frac{\sqrt{2 L (f(\vec{x}(0)){-}f^{\star})}}{\cos(\theta)\sqrt{T}} 
    \end{align}
     where the corresponding optimal $\gamma$ and $\kappa$ are
  \begin{align} \label{eq:inProof:Type-1-OP-sol2}
            \gamma^{\star}{=}\sqrt{\frac{2(f(\vec{x}(0)){-}f^{\star})}{LT}} \text{ and } 
            \kappa^{\star} {=}  \frac{\sqrt{ L (f(\vec{x}(0)){-}f^{\star})}}{\cos(\theta) \sqrt{2T}}.
    \end{align}

   ii) If the step-size $\gamma$ is chosen as $\gamma^{\star}$ in Equation~\eqref{eq:inProof:Type-1-OP-sol2} then  
    \begin{align} \label{eq:inThOptBou}
       ||\nabla f(\vec{x}^{\star}(T)) ||\leq \epsilon^{\star} 
    \end{align}
   where $\epsilon^{\star}$ is as in Equation~\eqref{eq:inProof:Type-1-OP-sol1} and
    \begin{align} \label{eq:inThOpt}
        \vec{x}^{\star}(T)= \underset{\vec{x}\in \{ \vec{x}(t) |t=0,\ldots,T\} }{\text{argmin}} || \nabla f(\vec{x})||.
    \end{align}
  \end{corollary}
  \begin{IEEEproof}
  i)
  First note that Problem~\eqref{eq:inProof:Type-1-OP} is convex, which can be seen by equivalently writing it as
  \begin{equation} \label{eq:inProof:Type-1-OP-2}
    \begin{aligned}
      & \underset{\kappa, \gamma}{\text{minimize}}
      & & \kappa + \frac{L}{2 \cos(\theta)} \gamma \\
      & \text{subject to}
      & & \left( \frac{f(\vec{x}(0))-f^{\star}}{ \cos( \theta) T  }\right) \frac{1}{\gamma} - \kappa  \leq  0, \\
      &&& \gamma, \kappa >0,
    \end{aligned}
  \end{equation} 
  and recalling that the reciprocal $1/\gamma$ is a convex function for $\gamma>0$.
  It can be checked that $\kappa^{\star}$ and $\gamma^{\star}$ satisfy the KKT condition with the Lagrangian multiplier $\lambda^{\star}{=}1$.

  ii) Follows directly from part i).
  \end{IEEEproof}
    In addition to minimizing the bound in Equation~\eqref{inProp:epsiloneq}, Corollary~\ref{prop:const_step-size:optimalAcc-given_itnum} gives insights into the convergence of  Iterations~\eqref{gradient_dir_quant}. 
   For example, when $T$ is fixed, then the upper bound in Equation~\eqref{eq:inProof:Type-1-OP-sol1} gets larger as $\theta$ decreases. 
   As a result, when the quantization set $\mathcal{D}$ becomes coarser then less accuracy can be ensured. 
   Moreover, the results show that $||\nabla f (\vec{x}^{\star}(T))||$ converges, in the worst case, at the rate $o(1/\sqrt{T})$ to $0$.

  In Corollary~\ref{prop:const_step-size:optimalAcc-given_itnum} we used the step-size  $\gamma^{\star}$  given in Equation~\eqref{eq:inProof:Type-1-OP-sol2}.
 To compute $\gamma^{\star}$ the optimal objective function value  $f^{\star}$ is needed, which is usually not available prior to solving Problem~\eqref{main_problem}. 
  However,  
  some bounds on the quantity $f(\vec{x}(0))-f^{\star}$ are often available. 
  Any such upper bound  $K\in \R$, with $K\geq f(\vec{x}(0))-f^{\star}$, can be used to obtain similar results as to those in Corollary~\ref{prop:const_step-size:optimalAcc-given_itnum} by replacing $f(\vec{x}(0))-f^{\star}$ by $K$.
 \begin{corollary}
    Take $T\in \N$ and $K \in \R$ such that $K \geq f(\vec{x}(0))-f^{\star}$.
    If we choose the step-size as $\gamma  = 2K/(LT)$ then
     \begin{align} \label{eq:inThOptBou}
       ||\nabla f(\vec{x}^{\star}(T)) ||\leq \frac{\sqrt{2 L K}}{\cos(\theta) \sqrt{T}},
    \end{align}
   where $\vec{x}^{\star}(t)$ is chosen as in Equation~\eqref{eq:inThOpt}.
 \end{corollary}

 We next demonstrate how the convergence results translate to Type-2 stopping conditions [Eq.~\eqref{Type-2}].


 \subsubsection{Stopping Condition of \emph{Type-2}}

    We now show that the Type-2 approximate optimality [Eq.~\eqref{Type-2}] can be achieved for any accuracy $\epsilon{>}0$. 
   The result is based on the following lemma. 
  \begin{lemma} \label{prop:fixed-stepsize-type-2}
      Suppose $f\in \mathcal{F}_L(\R^N)$, $\mathcal{D}$ is a $\theta$-cover, and the iterates $\vec{x}(t)$ are generated by Equation~\eqref{gradient_dir_quant},
       then: 
      
      \begin{enumerate}[a)]
      
       \item
      for any $\epsilon>0$, $\gamma(t)\in (0,\bar{\gamma})$, where $\bar{\gamma}=2\cos(\theta)\epsilon / L$, and $T\in \N$ such that $\vec{x}(T)\in \mathcal{X}(\epsilon)$, the following holds 
       \begin{align} \label{inProp:SC-T2-a}    
           f(\vec{x}(t)) \leq F(\epsilon)  + \frac{L}{2}  \gamma^2, ~~\text{ for all }~t\geq T,
       \end{align}
       where  $F:\R_+\rightarrow \R\cup \{\infty\}$ is given by
        \begin{align}  \label{inProp:SC-T2-a3}  
           F(\kappa) = \sup\{ f(\vec{x}) | \vec{x}\in \mathcal{X}(\kappa)\}.
        \end{align}
        There exists $\kappa>0$ such that for all $\epsilon\in [0,\kappa]$ (i) $\mathcal{X}(\epsilon)$ is bounded and $F(\epsilon) <\infty$. Moreover,
              $ \lim_{\epsilon \rightarrow 0^+} F(\epsilon) = f^{\star}$.     
       \item 
       if $f$ is $\mu$-strongly convex then we have  
       \begin{align}
                F(\epsilon) \leq f^{\star}+ \epsilon^2/(2\mu). 
       \end{align}    
      \end{enumerate}
    \end{lemma}      
    \begin{IEEEproof}
     The proof is provided in Appendix~\ref{App:prop:fixed-stepsize-type-2}.
    \end{IEEEproof}
    Lemma~\ref{prop:fixed-stepsize-type-2} is useful in obtaining Type-2 approximate optimality     as it connects the quantity $f(\vec{x}(t))-f^{\star}$  to Theorem~\ref{Prop:FSS-SC-Type-1} via the function  $F(\epsilon)$ in Equation~\eqref{inProp:SC-T2-a3}.
 In particular,  Part a) of Lemma~\ref{prop:fixed-stepsize-type-2} bounds $f(\vec{x}(t))$ by a constant that depends on $F(\cdot)$, where $F(\epsilon)$ converges to $f^{\star}$ as $\epsilon$ converges to $0$. 
  Therefore, by using the dependence of  $F(\epsilon)$ on $\mathcal{X}(\epsilon)$, defined in Equation~\eqref{eq:inProP-FSS-SC-Type-1-MCXep}, we can connect $f(\vec{x}(t))-f^{\star}$ to the convergence result in  Theorem~\ref{Prop:FSS-SC-Type-1} to ensure that the Type-2 stopping condition can be achieved for any $\epsilon>0$.  
 Part b) of Lemma~\ref{prop:fixed-stepsize-type-2} then illustrates how the upper bound on $f(\vec{x}(t))$ depending on $F(\epsilon)$ can be further improved when $f$ is $\mu$-strongly convex. 
  These ideas are formally illustrated in the following theorem.   
    \begin{theorem} \label{Theo:UnCon-FS-main}
             Suppose $f\in \mathcal{F}_L(\R^N)$, $\mathcal{D}$ is a $\theta$-cover, and $\vec{x}(t)$ are generated by Equation~\eqref{gradient_dir_quant}.
       Then  for any $\epsilon>0$:
      \begin{enumerate}[a)]
      
       \item there exists a step-size $\gamma>0$ and $T\in \N$ such that $f(\vec{x}(t))-f^{\star}<\epsilon$ for all $t\geq T$,
       
       \item moreover, if $f$ is $\mu$-strongly convex and
       $\gamma\in (0,\bar{\gamma})$ where
        $$\bar{\gamma}=\min \left\{ \frac{2\cos(\theta) \sqrt{\mu \epsilon}}{L},\sqrt{\frac{\epsilon}{L}} \right\}$$
       then $f(\vec{x}(t))-f^{\star} \leq \epsilon$ for all $t\geq T$ where
         \begin{align}  \label{eq:inProp:-conStep-a-Drec-obj}
                  T\leq \left\lceil \frac{2(f(\vec{x}(0))-f^{\star})}{\gamma(2 \cos(\theta) \sqrt{\mu \epsilon}-L\gamma)} \right\rceil.
         \end{align}  
       \end{enumerate}       
    \end{theorem}
    \begin{IEEEproof}
      a) The result follows directly from Lemma~\ref{prop:fixed-stepsize-type-2}-a) and Theorem~\ref{Prop:FSS-SC-Type-1}-a).
      
      b) Since $\gamma\in (0,2\cos(\theta) \sqrt{\mu \epsilon/L})$ it follows from Theorem~\ref{Prop:FSS-SC-Type-1}-a) there exists $T$ bounded as in Equation~\eqref{eq:inProp:-conStep-a-Drec} such that $\vec{x}(T)\in \mathcal{X}(\sqrt{\mu \epsilon})$. From Lemma~\ref{prop:fixed-stepsize-type-2}-a), for all $t\geq T$ that 
      \begin{align*}
         f(\vec{x}(t))- f^{\star} \leq&  F(\sqrt{\mu \epsilon}) - f^{\star} + \frac{L}{2}\gamma^2 
             \leq  \frac{\epsilon}{2}+\frac{\epsilon}{2}=\epsilon,
      \end{align*}
      where the second inequality follows from Lemma~\ref{prop:fixed-stepsize-type-2}-b) and that $\gamma\leq \sqrt{\epsilon/L}$.  
    \end{IEEEproof}

    Theorem~\ref{Theo:UnCon-FS-main} shows that when $\mathcal{D}$ is a $\theta$-cover the Type-2 optimality condition [Eq.~\eqref{Type-2}] can be achieved in finite number of iterations.
     For general functions, the theorem does not provide a step-size $\gamma_{\epsilon}$ that can achieve any particular $\epsilon$-accuracy, even though such $\gamma_{\epsilon}$ always exist.
 This is challenging in general, as it can be difficult to bound the function $F(\cdot)$ for general convex functions $f(\cdot)$.  
 Nevertheless, part b) of the proof shows that when $f$ is $\mu$-strongly convex, then a range of step-sizes that ensure a given $\epsilon>0$ accuracy is provided. 
   Moreover, when $f$ is $\mu$-stongly convex then we can obtain similar bound on number of iterations needed to achieve that $\epsilon$-accuracy as in Equation~\eqref{eq:inProp:-conStep-a-Drec} in Theorem~\ref{Prop:FSS-SC-Type-1}-a).


  \subsection{Diminishing Step-Size} \label{sec:Convergence:DimStep}

    We now consider the diminishing step-size case.
    The following result shows that the step-sizes can actually be chosen so Iterates~\eqref{gradient_dir_quant}  converges to the optimal solution to Problem~\eqref{main_problem}.

  \begin{theorem} \label{prop:dim_step_size}
   Suppose that $f\in \mathcal{F}_L(\R^N)$, $\mathcal{D}$ is a $\theta$-cover, and that the iterates $\vec{x}(t)$ are generated by Equation~\eqref{gradient_dir_quant}.
    If the step-sizes $\gamma(t)\geq 0$ are chosen so that  $\lim_{t\rightarrow 0} \gamma(t)=0$  and $\sum_{t=0}^N \gamma(t) =\infty$ 
    then  $\lim_{t\rightarrow \infty } \dist (\vec{x}(t),\mathcal{X}^{\star}) =  0$.
  \end{theorem}
  \begin{IEEEproof} 
  The proof is provided in Appendix~\ref{App:prop:dim_step_size}. 
  \end{IEEEproof}
  
    The step-size choice in the theorem  is necessary to ensure that  $\lim_{t\rightarrow \infty } \dist (\vec{x}(t),\mathcal{X}^{\star}) =  0$ for all $f\in \mathcal{F}_L(\R^N)$, consider the scalar function $f$ defined in Equation~\eqref{eq:flowerbound}.

 Theorem~\ref{prop:dim_step_size} shows that when $\mathcal{D}$ is a $\theta$-cover then there exists a step-size rule such that every limit point of the quantized gradient methods is an optimal solution to Problem~\eqref{main_problem}.
 A particular implication of this result is that every $\theta$-cover is a proper quantization, see Definition~\ref{def:proper quantization}.
 Therefore, Theorem~\ref{prop:dim_step_size}  actually proves one direction of the equivalence  established in 
Theorem~\ref{prop:equivalenceBetweenPQandTC}.

\section{Numerical Illustrations} \label{sec:NumericalResults}
 We now illustrate how the studied algorithms perform on two of the application examples discussed in Section~\ref{sec:Related Background}. We compare the numerical performance with some of the theoretical results in the paper and with algorithms that use perfect communication with no quanitzation.

   \subsection{TCP Flow Control with Binary Feedback}

\begin{figure}
    \centering
    \begin{subfigure}[b]{0.5\columnwidth}
        \includegraphics[width=\textwidth]{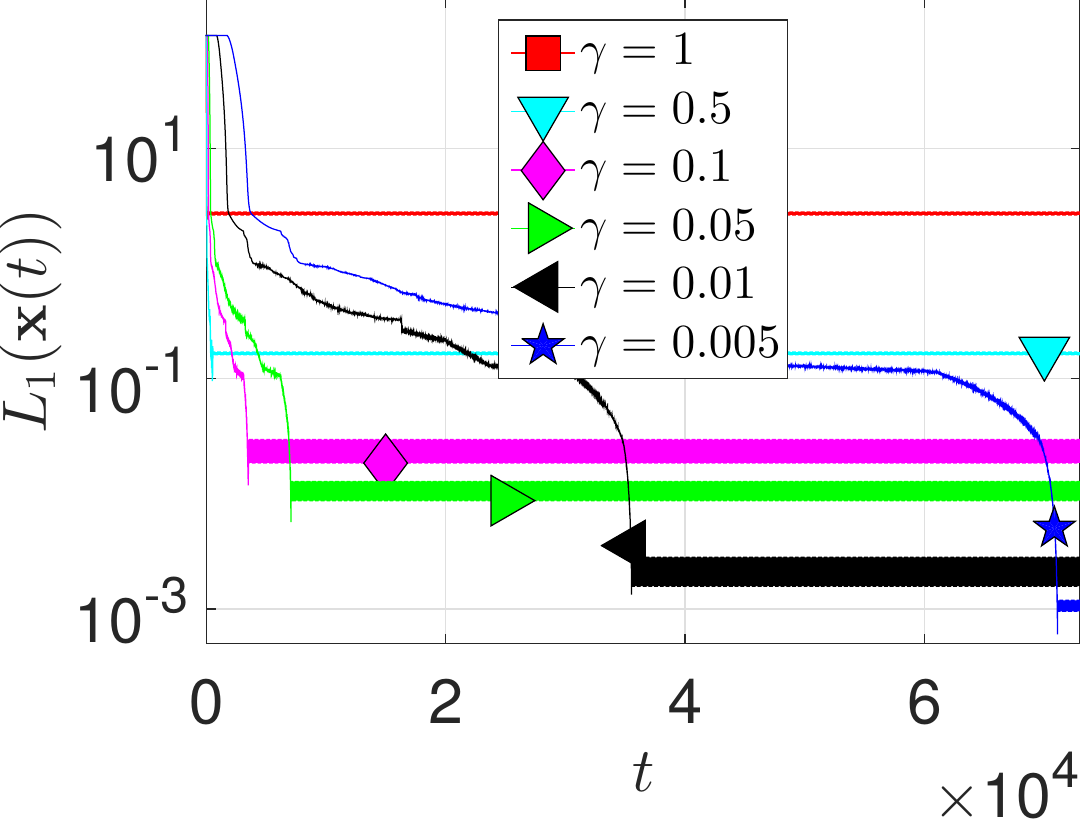}
        \caption{Optimality~condition~[Eq.~\eqref{eq:closedFun}]}
        \label{fig:TCP-A} 
    \end{subfigure}~
    \begin{subfigure}[b]{0.5\columnwidth}
        \includegraphics[width=\textwidth]{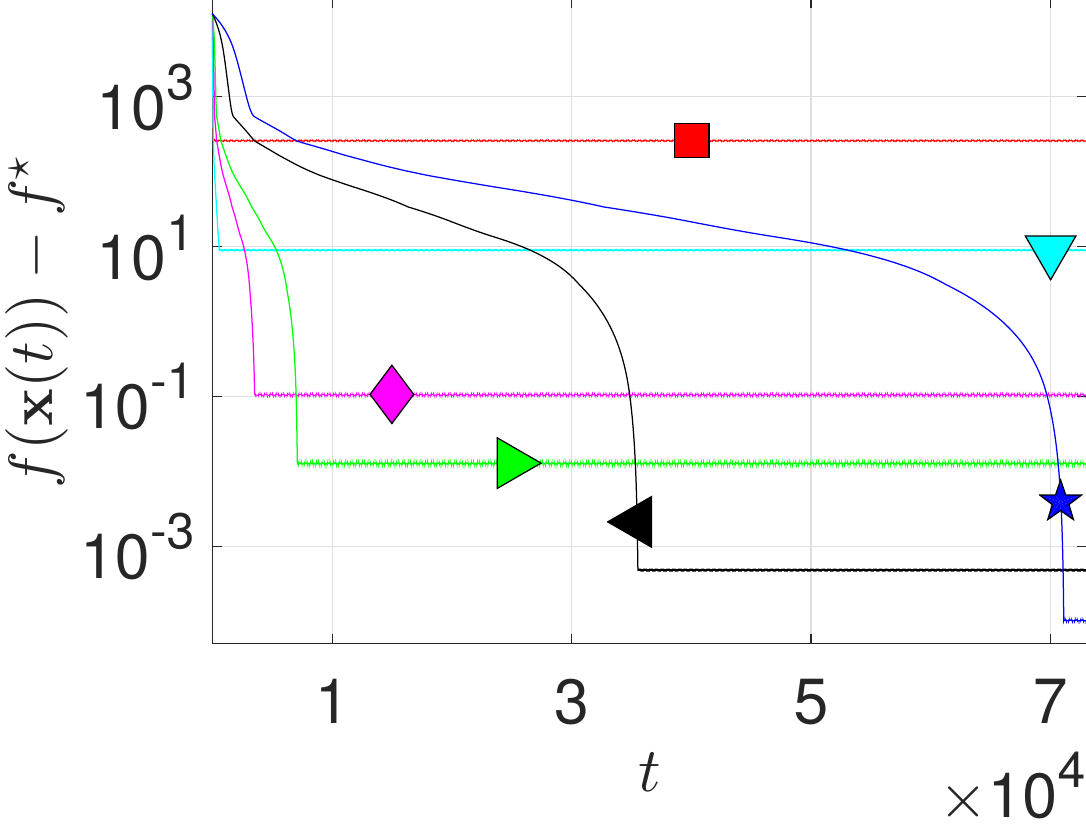}
        \caption{Dual objective function}
        \label{fig:TCP-B}
    \end{subfigure} 
    \begin{subfigure}[b]{0.5\columnwidth}
        \includegraphics[width=\textwidth]{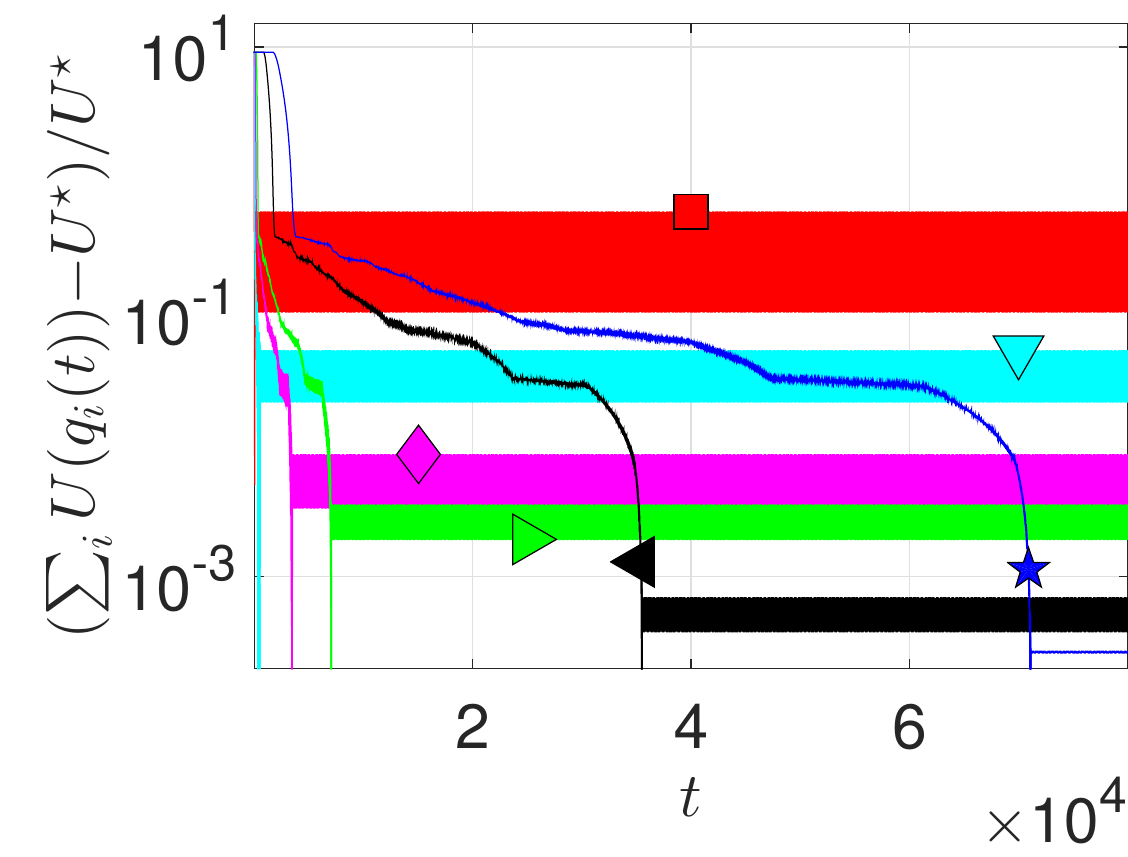}
        \caption{Relative primal objective value}
        \label{fig:TCP-C}
    \end{subfigure}~
    \begin{subfigure}[b]{0.5\columnwidth}
        \includegraphics[width=\textwidth]{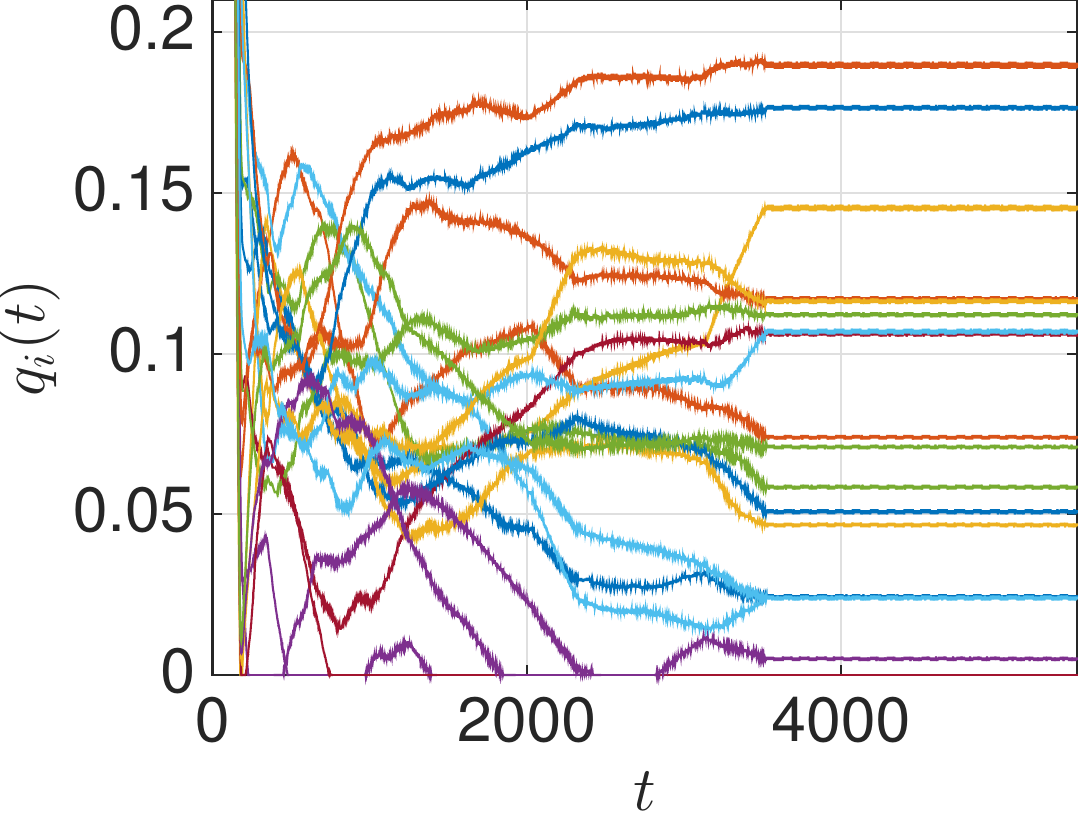}
        \caption{Allocation $\vec{q}(t)$ when $\gamma{=}0.1$}
        \label{fig:TCP-D}
    \end{subfigure} 
    \caption{Optimality measure $L_1 (\cdot)$ from Equation~\eqref{eq:closedFun} (Fig.~\ref{fig:TCP-A}), dual (Fig.~\ref{fig:TCP-B}), and primal (Fig.~\ref{fig:TCP-C}) objective function value at every iteration for different step-sizes. The resource allocations $q_i(t)$ (Fig.~\ref{fig:TCP-D}) at every iteration when $\gamma=0.1$.}
    \label{fig:TCP}
\end{figure}

  We illustrate the convergence of Iterates~\eqref{eq:prop:dim_step_size-Project-rec-1} on the TCP Flow Control in Section~\ref{subsec:Application_RA}. 
   We consider a network with $S=20$ sources and $N=100$ links. 
   We use the same utility functions as in Experiment~1 of~\cite[Section~VI-B]{low_1999}, i.e.,
    $U_s(q_s) = 1000 \log( 1+q_s), \text{ for } s\in \mathcal{S}$.
   Similarly as in~\cite{Wei_2013}, we generate the network matrix $\vec{A}$ [Eq.~\eqref{eq:TCPcontrol_A_matrix}] randomly so each entry of $\vec{A}$ is $1$ with the probability $1/2$ and $0$ otherwise.
    We set $c_l=1$ for all $l\in \mathcal{L}$. The local constraint of each source  
     is $[0,1]$. We consider the step-sizes $\gamma=0.005$ , $0.01$, $0.05$, $0.1$, $0.5$, and $1$. 
   Note that in the figures described below,  some of lines that appear to be thick lines actually show some fluctuations. 
    
    Fig.~\ref{fig:TCP-A} depicts the optimality measure $L_1(\vec{x}(t))$, see Equation~\eqref{eq:closedFun}  in Lemma~\ref{Lemma:OptCon}. 
   From Lemma~\ref{Lemma:OptCon},  $\vec{x}\in \R_+^N$ is an optimal solution to Problem~\eqref{main_problem} if and only if $L_1(\vec{x})=0$.  
  For all step-sizes $\gamma$, the measure $L_1(\vec{x}(t))$ converges to some small error floor and then  fluctuates slightly there.  
  For smaller step-sizes $\gamma$ the optimality measure $L_1(\vec{x}(t))$ converges to smaller values, roughly to $2.6$, $0.16$, $0.019$, $0.009$, $0.002$, $0.001$ for $\gamma=1$, $0.5$, $0.01$, $0.005$, $0.001$, and $0.0005$. 
  These results show that the step-size choices $\gamma$ in Theorem~\ref{Prop:FSS-SC-Type-1-Cons}  are conservative.\footnote{The parameters in Theorem~\ref{Prop:FSS-SC-Type-1-Cons}  for this problem are the parameters used are $\alpha=1$, $B=\sqrt{N}$, $\mu=250$, $\bar{S}=58$, $\bar{L}=17$ (see Lemma~\ref{lemma:multiple_resources_CDLC-TCP}).} 
  For example, to that ensure $\epsilon=0.1$ accuracy in Theorem~\ref{Prop:FSS-SC-Type-1-Cons} the step-sizes should be $\gamma\in (0,0.00002)$ but in for this example the step-size chooses $\gamma\leq0.1$ achieve the $\epsilon=0.1$ accuracy. 
   
   Figs.~\ref{fig:TCP-B} and~\ref{fig:TCP-C} depict the dual and primal objective function values at every iteration.
  The figures demonstrate a similar convergence behaviour of the primal/dual objective function values as in the optimality measure $L_{1}(\vec{x}(t))$ in Fig.~\ref{fig:TCP-A}. 
  For the dual objective value these results agree with the results in~ Theorem~\ref{Theo:UnCon-FS-main-Cons}.
Finally, Fig.~\ref{fig:TCP-D} illustrates the convergence of the data rate allocation to each source when $\gamma$.
 The results show that the all the data rate allocations converges after roughly 3500 iterations and then fluctuate slightly there.

 \subsection{Task Allocation}

 We now illustrate the performance of the quantized gradient methods on the Task Allocation Problem~\eqref{RA} from Section~\ref{subsec:Application_RA} with $K{=}4$ machines and $N{=}2$ tasks. 
   For each machine $i{=}1,2,3,4$ we have the cost function $C_i(\vec{q}_i)= a_i \vec{q}_{i,1}^2 + b_i\vec{q}_{i,2}^2$ where $a_i$ and $b_i$ are uniform random random variables on the interval $[1,5]$.  
    The private constraint of machine $i=1,2,3,4$ is $\mathcal{Q}_i=\{ (x,y) \in \R^2 | x,y\geq 0, x+y\leq 3\}$.
    Clearly,  $-C_i$ are strongly concave with concavity parameters  $\mu_i=\min\{ a_{i}, b_{i} \}\geq 1$. 
 It can be verified that the dual gradient is $L$-Lipschitz continuous with $L= 4/\mu$, where $\mu=\min\{\mu_1,\ldots, \mu_4 \}$. 
  The step-size is $\gamma=0.1$ and the initialization is $\vec{x}(0)=(0,0)$ (recall that $\vec{x}$ is the dual variable).
  We use the quantization set $\mathcal{D}$ from Example~\ref{Examples:theta-cover-2} when $2$, $3$, and $4$ bits are communicated per iteration, i.e., when $|\mathcal{D}|=4,8,16$, see Remark~\ref{remark:nr_of_bits}.
 
   Fig. \ref{fig:animals} depicts the norm of the gradient and the primal objective function at every iteration of the algorithm. 
      The norm of the gradient $||\nabla f||$ reaches the accuracy $\epsilon=0.1$ in roughly $51$, $56$, and $65$ iterations using $204$, $168$, $130$ bits when $2$, $3$, and $4$ bits are communicated per iteration, respectively.
      We compare the results to Iterations~\eqref{gradient_alg} and~\eqref{gradient_dir_alg} where no quantization is done, i.e., infinite bandwidth is used.
      Fig.~\ref{fig:TA-A}  shows that by using $4$ bits per iteration, the results achieved by QGM are almost as good as when the full gradient direction is communicated using Iterations~\eqref{gradient_dir_alg}.
      However, the QGMs do not perform as well as  Iterations~\eqref{gradient_alg}; this is to be expected, since in Iterations~\eqref{gradient_alg} the full direction and magnitude of the gradient is known. 
      These results illustrate that we can dramatically reduce the number of bits communicated without sacrificing much in performance.

\begin{figure}
    \centering
    \begin{subfigure}[h]{0.5\columnwidth}
        \includegraphics[width=\textwidth]{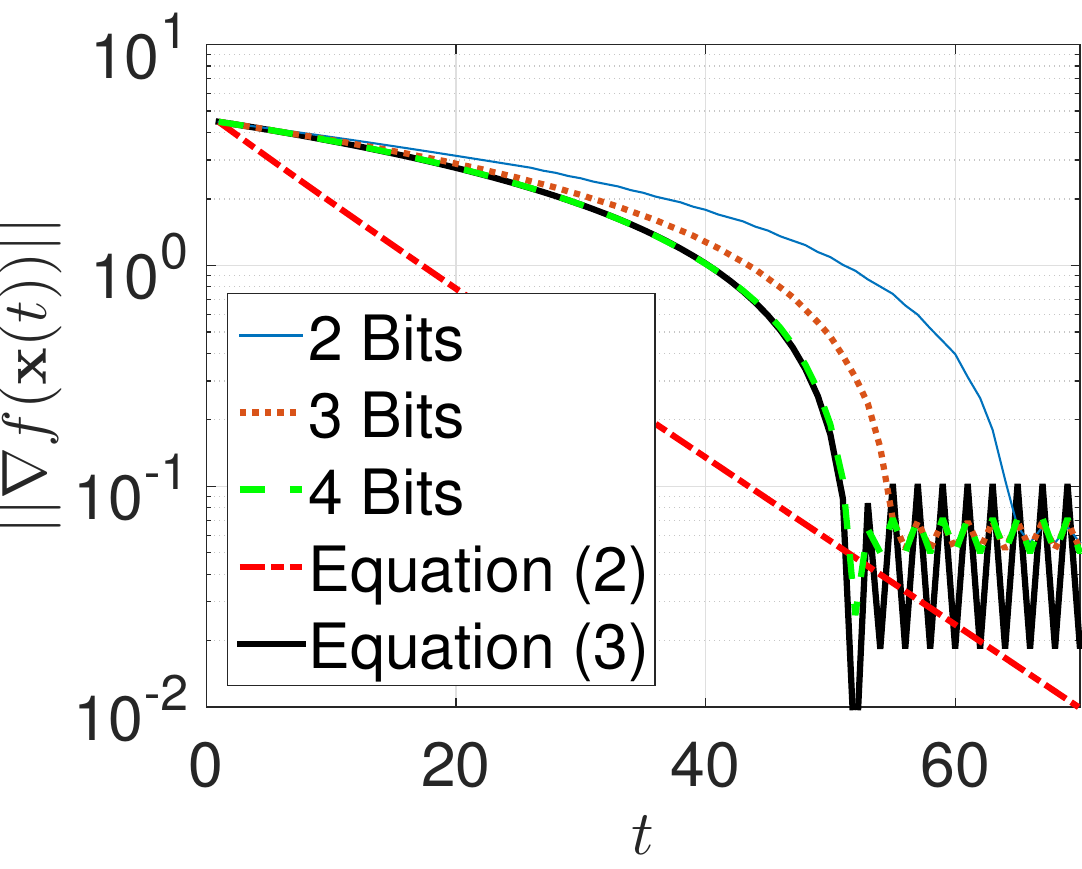}
        \caption{Dual gradient}
        \label{fig:TA-A} 
    \end{subfigure}~
    \begin{subfigure}[h]{0.5\columnwidth}
        \includegraphics[width=\textwidth]{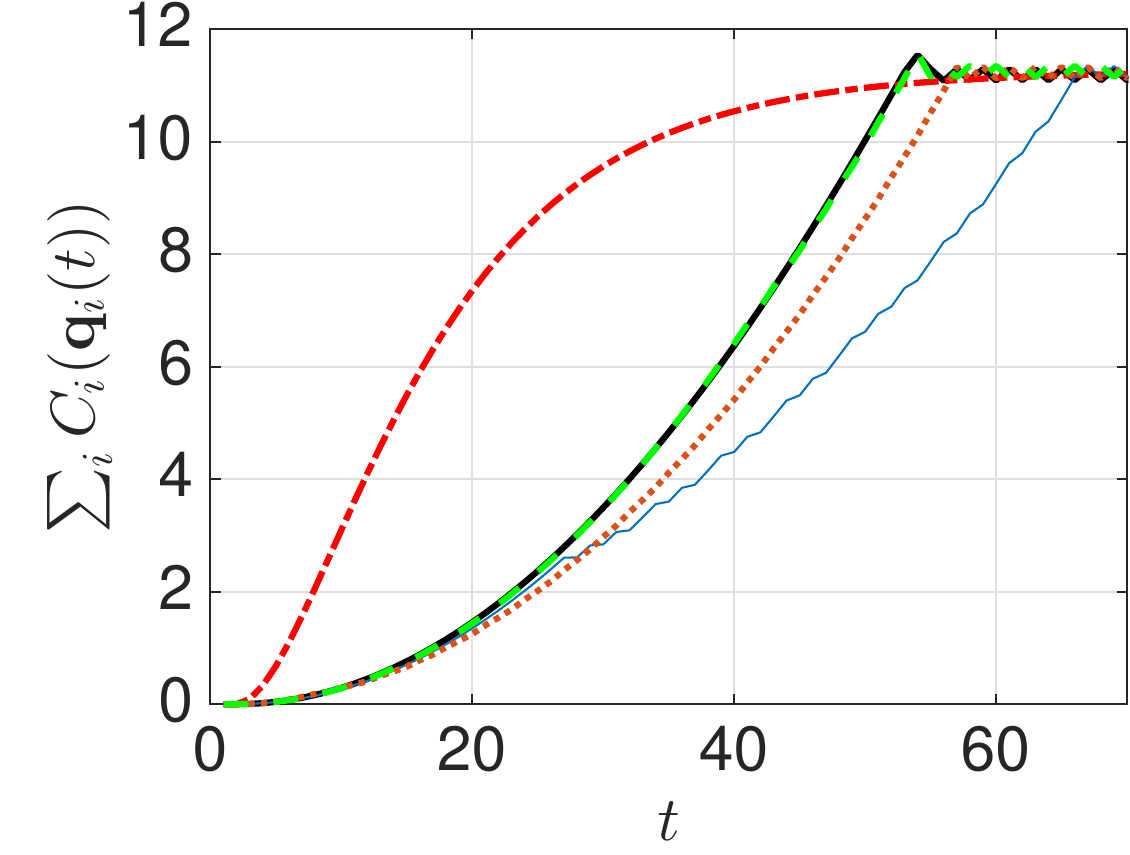}
        \caption{Objective function}
        \label{fig:TA-B}
    \end{subfigure} 
    \caption{The dual gradient (Fig.~\ref{fig:TA-A}) and the primal objective function value (Fig.~\ref{fig:TA-A}) at every iteration when $2$, $3$, and $4$ bits are communicated per iteration. The results are compared with the algorithms in Equations~\eqref{gradient_alg} and~\eqref{gradient_dir_alg}.}
    \label{fig:animals}
\end{figure}

  \section{Conclusions and future work} \label{sec:Conclusions}

     This paper studied gradient methods where the gradient direction is quantized at every iteration of the algorithm.     Such methods are of interest, for example, in distributed optimization where the gradient can often be measured but has to be communicated to accomplish the algorithm.
    An instance of such a procedure is dual decomposition, where primal problems that are scattered between different entities  are solved in a  distributed fashion by performing a gradient descent on the dual problem, see Section~\ref{sec:Related Background} for examples.   Our results show that a variant of the projected dual descent taking the sign of the dual gradients, i.e.,
    \begin{align} \label{eq:concl-signgrad} \vec{x}(t{+}1)=    \left[\vec{x}(t) - \gamma(t)~ \sign(\nabla f(\vec{x}(t)) \right]_{\mathcal{X}},\end{align}
 converges to the optimal solution under mild conditions on $f$ when $\mathcal{X}=\R^N$ and $\mathcal{X}=\R_+^N$, i.e., for dual problems associated with primal problems with equality and inequality constraints, respectively.  
 Therefore, when different entities maintain the components of the variable $\vec{x}$ then each entity only needs to broadcast one bit per iteration to ensure convergence to the optimal solution. 
    Our results also show that when a single entity maintains $\vec{x}$ then the minimal quantization has cardinality $N+1$; for smaller quantizations there exists an optimization problem that the quantized gradient methods cannot solve. 
  Therefore, only $\log_2(N+1)$ bits/iteration are communicated  instead of $N$ bits/iteration as when the components of $\vec{x}$ are maintained by different entities  in Equation~\eqref{eq:concl-signgrad}. 
     We also connect fineness of the quantization to the convergence rate of the algorithm to the available bandwidth (bits/iteration). 
   The convergence rate improves as the bandwidth is increased.

    Future work will consider how to additionally quantize the magnitude of the gradient to get a better trade-off between convergence rate and available bandwidth.
    Moreover, it is interesting to see if the results can be generalized to nonsmooth optimization problems.

  \appendices

\section{Proof of Theorem~\ref{prop:equivalenceBetweenPQandTC}} \label{App:prop:equivalenceBetweenPQandTC}
   \begin{IEEEproof}
Let us start by showing by contradiction that $\mathcal{D}$ being a proper quantization for the problem class $\mathcal{F}_L(\R^N)$ implies that there exists $\theta\in [0,\pi/2)$ such that $\mathcal{D}$ is a $\theta$-cover.
 Suppose there does not exists such $\theta$. 
  Then $\min_{\vec{a}\in\mathcal{S}^{N-1}} \max_{\vec{d}\in\mathcal{D}} \cos(\ang(\vec{a},\vec{d}))\leq 0$ since the function  $g(\vec{a})=\max_{\vec{d}\in\mathcal{D}} \cos(\ang(\vec{a},\vec{d}))=\max_{\vec{d}\in\mathcal{D}} \langle \vec{a},\vec{d}\rangle$ is continuous and $\mathcal{S}^{N-1}$ is compact. 
    Therefore, there exists $\vec{a}\in \mathcal{S}^{N-1}$ such that $\cos(\ang(\vec{a},\vec{d}))\leq 0$ for all $\vec{d}\in \mathcal{D}$.
  In particular, we have for all $\vec{d}\in \mathcal{D}$ that
    $\langle \vec{a},\vec{d} \rangle =  \cos(\ang(\vec{a},\vec{d})) \leq 0$.
  By choosing $\vec{x}(0)= \vec{a}$, using Iterations~\eqref{gradient_dir_quant} and Cauchy-Schwarz inequality we conclude that for all $t\in \N$ 
    \begin{align}
     ||\vec{x}(t)||\geq  \langle \vec{a}, \vec{x}(t)\rangle =&  \langle \vec{a}, \vec{a} \rangle -\sum_{i=0}^{t-1} \gamma(t)  \langle \vec{a}, \vec{d}(t)\rangle  
          \geq 1, \notag
    \end{align}
   where the inequality follows from the fact that $||\vec{a}||=1$ and  that for all  $\vec{d}\in \mathcal{D}$ we have 
  $\langle \vec{a},\vec{d}\rangle \leq 0$. 
  If we choose $f(\vec{x})=(L/2)||\vec{x}||^2$, then $f\in \mathcal{F}_L(\R^N)$  and 
 $f$ has the unique optimizer $\vec{x}^{\star}=\vec{0}$, but
   $ \dist(\vec{x}(t),\mathcal{X}^{\star}) = || \vec{x}(t) || \geq 1$,
  for all $t\in \N$. 
  Since $ \dist(\vec{x}(t),\mathcal{X}^{\star}) \geq 1$ for all $\vec{d}(t)\in \mathcal{D}$ and $\gamma(t){\in} \R_+$, we can conclude that $\mathcal{D}$ is not a proper quantization.

   The fact that $\mathcal{D}$ being a $\theta$-cover implies that $\mathcal{D}$ is a proper quantization, follows from Theorem~\ref{prop:dim_step_size} in Section~\ref{sec:Convergence:DimStep}, where we showed that for all $f\in \mathcal{F}_L(\R^N)$ we can choose $\vec{d}(t)\in \mathcal{D}$ and $\gamma(t) \in \R_+$ such that $\lim_{t\rightarrow \infty} \dist(\vec{x}(t),\mathcal{X}^{\star})=0$.
 \end{IEEEproof}
 
 \section{Proof of Theorem~\ref{Theorem:MinimalProperQuant}} \label{APP:Theorem:MinimalProperQuant}
  \begin{IEEEproof}
  First consider the case where either $|\mathcal{D}|<N$ or $|\mathcal{D}|=N$ and the elements of $\mathcal{D}$ are linearly dependent. 
 Then  $\Span(\mathcal{D})$ is a proper subspace of $\R^{N}$, so there exists a normal $\vec{a}\in \mathcal{S}^{N-1}$ such that $\cos(\ang(\vec{a}, \vec{d}))=\langle \vec{a},\vec{d} \rangle \leq 0$ for all $\vec{d}\in \Span(\mathcal{D})$. 
   Since $\mathcal{D}\subseteq \Span(\mathcal{D})$, $\mathcal{D}$ is not a $\theta$-cover for any $\theta\in [0,\pi/2)$ and the result  follows from Theorem~\ref{prop:equivalenceBetweenPQandTC}.

 Let us next consider the other case, where $|\mathcal{D}|=N$ and the vectors of $\mathcal{D}$ are linearly independent, i.e., $\Span(\mathcal{D})=\R^N$. 
 Define $\vec{D}\in\R^{N\times N}$ such that  for $i=1,\ldots, N$ row $i$ in $\vec{D}$ is the $i$-th elemnt of $\mathcal{D}$, where the elements have some arbitrary  order. 
 Then $\vec{D}$ is invertible and we can choose $\vec{a}=-\vec{D}^{-1} \vec{1}$ where $\vec{1}\in\R^N$ is a vector of all ones.
  Then we have for $i=1,\ldots,N$ that $\langle \vec{d}_i, \vec{a}\rangle = - \vec{d}_i \vec{D}^{-1} \vec{1}=-1$.
  Hence, as in the previous case,  we get that  $\langle \vec{a},\vec{d} \rangle \leq 0$ for all $\vec{d}\in \mathcal{D}$ implying that $\mathcal{D}$ can not be a $\theta$-cover for any $\theta\in[0,\pi/2)$, and the result follows from Theorem~\ref{prop:equivalenceBetweenPQandTC}.
 \end{IEEEproof}


   \section{Proof of Lemma~\ref{Lemma:OptCon}} \label{App:Lemma:OptCon}
   \begin{IEEEproof}
     Since $f$ and $\mathcal{X}$ are convex, $\vec{x}\in \mathcal{X}^{\star}$ if and only if the KKT optimality conditions hold for $\vec{x}$~\cite[Section 5.5]{convex_boyd}.
     It can be checked that since $\mathcal{X}=\R_+^N$ the optimal dual variable associated to $\vec{x}\in \mathcal{X}^{\star}$ is $\boldsymbol\lambda= \nabla f(\vec{x})$. Therefore, the KKT conditions reduce to the following three condition holding for all $i=1,\ldots,N$:
     (i) $\nabla_i f(\vec{x}) \vec{x}_i =0$, (ii) $\nabla_i f(\vec{x}) \geq 0$, and (iii) $\vec{x}_i\geq 0$.
     We now show both directions of the proof.

     First assume that $L(\vec{x})=0$. We show that (i), (ii), and (iii) hold, so $\vec{x}\in \mathcal{X}^{\star}$.
     We have $\vec{x}=\lceil \vec{x}-\alpha \nabla f(\vec{x})\rceil^+$ or $\vec{x}_i=\lceil \vec{x}_i-\alpha \nabla_i f(\vec{x})\rceil^+$, for $i=1,\ldots, N$.
     So (i) holds because if $\vec{x}_i\neq 0$ then $\nabla_i f(\vec{x})=0$ and if $\nabla_i f(\vec{x})\neq 0$ then $\vec{x}_i= 0$.
     Similarly, (ii) holds because  if $\vec{x}_i\neq 0$ then $\nabla_i f(\vec{x})=0$ and if $\vec{x}_i=0$ then $\nabla_i f(\vec{x})\geq 0$.
    Finally, (iii) holds because $\lceil \cdot \rceil^+$ is the projection to $\R_+$.
    
    Now assume $\vec{x}\in \mathcal{X}^{\star}$ so (i), (ii), and (iii) above hold. 
    If $\vec{x}_i=0$ for some $i=1,\ldots, N$ then $\vec{x}_i=\lceil \vec{x}_i-\alpha \nabla_i f(\vec{x})\rceil^+$ since by (ii) $\nabla_i f(\vec{x})\geq 0$.
    Otherwise, if $\vec{x}_i>0$, then $\nabla_i f(\vec{x})=0$ by (i) so $\vec{x}_i=\lceil \vec{x}_i-\alpha \nabla_i f(\vec{x})\rceil^+$.
  \end{IEEEproof}

 \section{Proof of Lemma~\ref{lemma:binary_feedback_descent_lemma-Pojected}} \label{App:lemma:binary_feedback_descent_lemma-Pojected}
 \begin{IEEEproof} 
 For all $\beta\in [0,1]$ we have
 \begin{align}
   \frac{\beta \epsilon }{ \sqrt{N}} & \leq  \frac{\beta  }{\sqrt{N}} ||\vec{x}- \lceil \vec{x}-\alpha \nabla f(\vec{x}) \rceil^+ || \label{eq:LSEQ1}\\
   & \leq  \beta  ||\vec{x}- \lceil \vec{x}-\alpha \nabla f(\vec{x}) \rceil^+ ||_{\infty}  \label{eq:LSEQ2} \\
   & = \beta  |\vec{x}_{i}- \lceil \vec{x}_{i}-\alpha \nabla_{i} f(\vec{x}) \rceil^+ |  \label{eq:LSEQ3} \\  
   & \leq   |\vec{x}_{i}{-} \lceil \vec{x}_{i}-\alpha \beta B \sign(  \nabla_{i}  f(\vec{x})) \rceil^+ |  \label{eq:LSEQ4} 
 \end{align}
 where, in Equation~\eqref{eq:LSEQ3},
      $i=\text{argmax}_j  |x_{j}(t)- \lceil x_{j}(t)-\alpha \nabla_{j} f(x(t)) \rceil^+ |$,
Equation~\eqref{eq:LSEQ1} comes from the bound in Equation~~\eqref{eq:LemmaProDec-OC},
Equation~\eqref{eq:LSEQ2} comes from the equivalence of the $2$- and $\infty$-norms,
Equation~\eqref{eq:LSEQ3} comes by using the definition of $i$,
Equation~\eqref{eq:LSEQ4} comes from Equations~\eqref{eq:LemPro-1} and~\eqref{eq:LemPro-2} in Lemma~\ref{Lemma:projections} in Appendix~\ref{App:AddLemmas} and that $||\nabla f||\leq B$. 
 By taking $\beta= \gamma/ (\alpha B \sqrt{N})$ in Equation~\eqref{eq:LSEQ4} we get  
 \begin{align}
    \frac{\epsilon \gamma}{ \alpha B N} \leq& 
      |\vec{x}_i - \left\lceil  \bar{\vec{x}}_i   \right\rceil^+ |, \label{eq:proofPro-it-bound-1} 
 \end{align}
 where
    $\bar{\vec{x}} = \vec{x}{-} (\gamma/\sqrt{N}) \sign(\nabla f(\vec{x}))$. 
  Moreover, from Equation~\eqref{eq:LSEQ3} and the nonexpansiveness of projections, see~\cite[Proposition B.11-c)]{nonlinear_bertsekas}, we have
 \begin{align}
      \frac{ \epsilon}{\alpha \sqrt{N}} \leq |\nabla_i f(\vec{x})|.  \label{eq:proofPro-nable-bound}
 \end{align}
 Therefore, we have 
   \begin{align}
      \langle \nabla f(\vec{x}),  \vec{x} {-}  \lceil \bar{\vec{x}} \rceil^+  \rangle =& 
      \sum_{j=1}^N  \nabla_j f(\vec{x})(\vec{x}_j {-}    \lceil \bar{\vec{x}}_j \rceil^+), \label{eq:proofPro-Inn-1} \\
       {\geq}&  \nabla_i f(\vec{x})(\vec{x}_i  {-}   \lceil \bar{\vec{x}}_i \rceil^+), \label{eq:proofPro-Inn-2} \\
       {\geq}&\frac{\epsilon^2 \gamma}{  \alpha^2 B N^{3/2}} , \label{eq:proofPro-Inn-3}
 \end{align}
 where Equation~\eqref{eq:proofPro-Inn-2} comes from the fact that every component of the sum is nonnegative, see Equation~\eqref{eq:LemPro-3}  in Lemma~\ref{Lemma:projections} in Appendix~\ref{App:AddLemmas}, and Equation~\eqref{eq:proofPro-Inn-3}  comes from using the bound in Equations~\eqref{eq:proofPro-it-bound-1} and~\eqref{eq:proofPro-nable-bound}.
  Inequality~\eqref{eq:proofPro-Inn-3} and the descent lemma~\cite[eq.~(2.1.6)]{Book_Nesterov_2004} yield
 \begin{align*}
    f(\lceil \bar{\vec{x}}\rceil^+) 
        {\leq} & f(\vec{x}) {-} \langle \nabla f(\vec{x}), \vec{x} {-} \lceil \bar{\vec{x}} \rceil^+  \rangle  
      {+} \frac{L}{2} || \vec{x} {-} \lceil \bar{\vec{x}} \rceil^+||^2, \\ 
    {\leq}& f(\vec{x}) -\frac{\epsilon^2 \gamma}{ \alpha^2 B N^{3/2}}  {+}  \frac{L}{2} \gamma^2 ,
 \end{align*}    
 where the last term comes from the fact that $|| \vec{x} {-} \lceil \bar{\vec{x}} \rceil^+||\leq \gamma$ by the non-expansiveness of the projection.
\end{IEEEproof}

 \section{Proof of Lemma~\ref{prop:fixed-stepsize-type-2-Proj}}
 \label{App:prop:fixed-stepsize-type-2-Proj}
 \begin{IEEEproof} 
  \textbf{Step 1:} We prove by induction that Equation~\eqref{inProp:SC-T2-a-Proj} holds for all $t\geq T$.
  When $t=T$ then Equation~\eqref{inProp:SC-T2-a-Proj}     holds because $f(\vec{x}(T))\leq \bar{F}_{\alpha}(\epsilon)$ 
  by definition of $\bar{F}_{\alpha}$, see Equation~\eqref{inProp:SC-T2-a3-Proj}.
  Now suppose Equation~\eqref{inProp:SC-T2-a-Proj} holds for  $t\geq T$. We will show that~\eqref{inProp:SC-T2-a-Proj} also holds for $t+1$.
  Consider first the case when $\vec{x}(t)\in \bar{\mathcal{X}}_{\alpha}(\epsilon)$.  
  Then from~\cite[eq.~(2.1.6)]{Book_Nesterov_2004}
  \begin{align*}
   f(\vec{x}(t{+}1)) \leq& f(\vec{x}(t))  {-} \langle \nabla f(\vec{x}(t)), \vec{x}(t) {-} \lceil \bar{\vec{x}}(t) \rceil^+  \rangle   \\ 
        &~~~~~~~~~~~~~~~
        {+} \frac{L}{2} || \vec{x}(t) {-} \lceil \bar{\vec{x}}(t) \rceil^+||^2, \\
        \leq& \bar{F}_{\alpha}(\epsilon) + \frac{L}{2} \bar{\gamma}^2
  \end{align*}
   where  $\bar{\vec{x}}(t) = \vec{x}(t)-  (\gamma(t)/\sqrt{N}) \sign(\nabla f(\vec{x}(t)))$,   
   the second inequality comes from the fact that (i) that $f(\vec{x}(t))\leq \bar{F}_{\alpha}(\epsilon)$ since $\vec{x}(t)\in \bar{\mathcal{X}}_{\alpha}(\epsilon)$,
   (ii)  that the inner product term is non-negative because every term of the sum [Eq.~\eqref{eq:proofPro-Inn-1}] is non-negative following Equation~\eqref{eq:LemPro-3}  in Lemma~\ref{Lemma:projections}, and (iii) that $|| \vec{x}(t) {-} \lceil \bar{\vec{x}}(t) \rceil^+||\leq \bar{\gamma}$ because of the non-expansiveness of the projection $\lceil \cdot \rceil^+$, see~\cite[Proposition B.11-c)]{nonlinear_bertsekas}.
 Otherwise, if  $\vec{x}(t)\notin \bar{\mathcal{X}}_{\alpha}(\epsilon)$ then  $f(\vec{x}(t+1))\leq f(\vec{x}(t))$  by  Lemma~\ref{lemma:binary_feedback_descent_lemma-Pojected}, yielding the result.

  \textbf{Step 2:} We will prove that there exists $\kappa>0$  such that (i) $\bar{\mathcal{X}}_{\alpha}(\epsilon)$ is bounded set and (ii) $\bar{F}_{\alpha}(\epsilon)<\infty$ for all $\epsilon\in [0,\kappa]$.
  Part~(i) follows directly from Lemma~\ref{lemma:bounded_Xeps} in Appendix~\ref{App:AddLemmas}.  
  To prove part (ii), note that $\bar{\mathcal{X}}_{\alpha}(\epsilon)$ is closed set and also bounded for all $\epsilon\in [0,\kappa]$ for some $\kappa>0$ from part (i).
 In particular,   $\bar{\mathcal{X}}_{\alpha}(\epsilon)$ is compact set for all $\epsilon\in [0,\kappa]$ so  the supremum  in~\eqref{inProp:SC-T2-a3-Proj}   is attained and hence  $\bar{F}_{\alpha}(\epsilon)<\infty$.

   \textbf{Step 3:} We will prove that $\lim_{\epsilon\rightarrow 0^+} \bar{F}_{\alpha}(\epsilon)= f^{\star}$.
    In particular, we show that $\bar{F}_{\alpha}$ is continuous at $0$ which implies the result, since $\bar{F}_{\alpha}(0)=f^{\star}$.
    Take any sequence $(\epsilon_k)_{k\in \N}$ in $\R_+$ such that $\lim_{k\rightarrow \infty} \epsilon_k = 0$. 
    Then there exists $K\in \N$ and a sequence $(\vec{x}(k))_{k\in \N}$ such that $f(\vec{x}(k))=\bar{F}_{\alpha}(\epsilon)$ holds for all $k\geq K$, since $\bar{\mathcal{X}}_{\alpha}(\epsilon)$ is compact for all $\epsilon \in [0,\kappa]$, where $\kappa$ is chosen as in \textbf{Step 2}. 
    Moreover, by the definition of $\bar{\mathcal{X}}_{\alpha}(\epsilon)$ we have that $\lim_{k\rightarrow \infty} L_{\alpha}(\vec{x}(k))=0$. 
    Now since $L_{\alpha}$ is a continuous function we can conclude that for every limit point $\bec{x}$ of $(\vec{x}(k))_{k\in \N}$ it holds that $L_{\alpha}(\bec{x})= 0$, i.e., $\bec{x} \in \mathcal{X}^{\star}$ or $f(\bec{x})=f^{\star}$.
    Since $f(\bec{x})=f^{\star}$ holds for every limit point of $(\vec{x}(k))_{k\in\N}$ and $f$ is continuous we can conclude that 
      $\lim_{k\rightarrow \infty} f(\vec{x}(k)) = \lim_{k\rightarrow \infty} \bar{F}_{\alpha}(\epsilon_k)= f^{\star}$.   
%
\end{IEEEproof}

\section{Proof of Theorem~\ref{prop:dim_step_size-Project}} \label{App:subseq:Pro-DSS}
\begin{IEEEproof} 
 \textbf{Step 1:} We will prove by contradiction that for any $\alpha>0$
 \begin{equation} \label{eq:LSEQ-liminf}
  I := \liminf_{t \rightarrow\infty } L_{\alpha}(\vec{x}(t)) = 0. 
 \end{equation}
 Suppose, to the contrary, that $I>0$. 
 Choose $T$ such that $L_{\alpha}(\vec{x}(t))   \geq I/2$ and 
   $ \gamma(t) < \min \{ 1, \allowbreak   I^2/(4 L \alpha^2 B N^{3/2})  \}$
 for all $t\geq T$. 
 Then by Lemma~\ref{lemma:binary_feedback_descent_lemma-Pojected} we get for all $t\geq T$
 \begin{align}
    f(\vec{x}(t{+}1)) {\leq}& f(\vec{x}(t)) -\frac{I^2 \gamma(t)}{4\alpha^2 B N^{3/2}}  +  \frac{ L}{2}\gamma(t)^2 \label{eq:LemmaPQGM-Lyp-1}\\
                  {=}& f(\vec{x}(t)) {-}  \frac{I^2 \gamma(t)}{8\alpha^2 B N^{3/2}} 
                      {+} \hspace{-0.1cm}  \left( \hspace{-0.1cm}  \gamma(t) {-}  \frac{I^2}{4 L \alpha^2 B N^{3/2}} \hspace{-0.1cm}  \right) \hspace{-0.13cm} \frac{L}{2} \gamma(t) \notag \\
                   {\leq}& f(\vec{x}(t)) -  \frac{I^2 \gamma(t)}{8 \alpha^2 B N^{3/2}}. \label{eq:LemmaPQGM-Lyp-3}
 \end{align}
 Since Equation~\eqref{eq:LemmaPQGM-Lyp-3} holds for all $t\geq T$ we obtain 
 \begin{align}
  f(\vec{x}(t)) \leq f(\vec{x}(T)){-} \frac{I^2 }{8\alpha^2 B N^{3/2}} \sum_{\tau=T}^{t-1} \gamma(\tau),~\text{for $t\geq T$}.  \label{eq:LemmaPQGM-Cont-sum}
 \end{align}
Since $\gamma(t)$ is non-summable Equation~\eqref{eq:LemmaPQGM-Cont-sum} implies that 
 $\lim_{t\rightarrow \infty} f(\vec{x}(t))=-\infty,$
which contradicts the fact that $\mathcal{X}^{\star}$ is non-empty.
Therefore we can conclude that $I=0$, [Eq.~\eqref{eq:LSEQ-liminf}].

\textbf{Step 2:} We will prove that $\lim_{t\rightarrow \infty } f(\vec{x}(t))=f^{\star}$. 
Let $\epsilon>0$ be given. 
Choose $\kappa>0$ such that $\bar{F}_{\alpha}(\kappa)<f^{\star} +\epsilon/2$, where $\bar{F}_{\alpha}(\kappa)$ is defined in Equation~\eqref{inProp:SC-T2-a3-Proj} of Lemma~\ref{prop:fixed-stepsize-type-2-Proj}, such $\kappa$ exists since $\lim_{\epsilon\rightarrow 0^+} F(\epsilon) = f^{\star}$.
 Now choose $T$ such that $\vec{x}(T)\in \bec{\mathcal{X}}_{\alpha}(\kappa)$ [Eq.~\eqref{eq:inProP-FSS-SC-Type-1-MCXep-Proj}] and for all $t\geq T$ it holds that $\gamma(t) \leq \bar{\gamma}:=\sqrt{\epsilon/L}$, such $T$ exists because of Equation~\eqref{eq:LSEQ-liminf}  and that $\lim_{t\rightarrow \infty} \gamma(t)=0$. 
 Then from Equation~\eqref{inProp:SC-T2-a-Proj} in Lemma~\ref{prop:fixed-stepsize-type-2-Proj} we have for all $t\geq T$ that
 $$ f(\vec{x}(t))-f^{\star}\leq \bar{F}_{\alpha}(\kappa)-f^{\star}+ \frac{L }{2} \bar{\gamma}^2 \leq \frac{\epsilon}{2}+\frac{\epsilon}{2} =\epsilon. $$

  \textbf{Step 3:} We will prove  that the sequence $\vec{x}(t)$ is bounded.
 Take $\kappa>0$ such that $\bar{F}_{\alpha}(\kappa)<\infty$, where $\bar{F}_{\alpha}$ is defined in Equation~\eqref{inProp:SC-T2-a3-Proj}, such $\kappa$ exists by Lemma~\ref{prop:fixed-stepsize-type-2-Proj}.   
 From Equation~\eqref{eq:LSEQ-liminf},  $\vec{x}(t)\in\bec{\mathcal{X}}_{\alpha}(\kappa)$ holds for infinitely many $t\in \N$.
 Let $\vec{x}(t_k)$, with $k\in \N$, be the subsequence of all $\vec{x}(t)\in\bar{\mathcal{X}}_{\alpha}(\kappa)$. 
  Choose $T\in\N$ such that $\gamma(t)\leq \epsilon^2/(L\alpha^2 B N^{3/2})$ for all $t\geq T$.
 Then, by following the same steps as used to obtain Equations~\eqref{eq:LemmaPQGM-Lyp-3} and~\eqref{eq:LemmaPQGM-Cont-sum} and using the fact that $ f(\vec{x}(t_k))\leq  \bar{F}_{\alpha}(\kappa)$, we have for every $k\in \N$ such that $t_k\geq T$  and all $t\in \N$ such that $t_k<t<t_{k+1}$ that 
 \begin{align*}
   f(\vec{x}(t)) 
     \leq& \bar{F}_{\alpha}(\kappa)-\frac{\epsilon^2}{2 \alpha^2 B N^{3/2}} \sum_{\tau=t_k}^{t} \gamma(\tau). 
 \end{align*}
 Therefore, since $f^{\star}\leq f(\vec{x}(t))$ we have that
 \begin{align} \label{eq:ThConDimBounded}
      \sum_{\tau=t_k}^{t_{k+1}} \gamma(t) \leq \frac{2 \alpha B N^{3/2}}{\epsilon^2} ( \bar{F}_{\alpha}(\kappa)-f^{\star}).
   \end{align}
  We also have from  Lemma~\ref{prop:fixed-stepsize-type-2-Proj} that $\bec{\mathcal{X}}_{\alpha}(\kappa)$ is bounded so there exists $A\in \R_+$ such that $||\vec{x}||\leq A$ for all $ \vec{x} \in\bec{\mathcal{X}}_{\alpha}(\kappa)$.
 As a result, for all $t\geq T$ and $k(t)=\max\{ k\in \N | t_k\leq t \}$ we get 
 \begin{align*}
    ||\vec{x}(t)||  \leq&  \sum_{\tau=t_{k(t)}}^{t-1} ||\vec{x}(\tau{+}1){-}\vec{x}(\tau)|| {+} ||\vec{x}(t_k)||   
                        {\leq} \sum_{\tau=t_{k(t)}}^t \gamma(\tau)  {+}A 
 \end{align*}
 where the first inequality comes by writing  $\vec{x}(t)$  as a telescoping series starting at $\vec{x}(t_k)$ together with the triangle inequality and
 the second inequality  comes from the relation
 $$||\vec{x}(t{+}1)-\vec{x}(t)||=\left|\left|\left\lceil \bar{\vec{x}}(t) \right\rceil^+ -\vec{x}(t)\right|\right|\leq \gamma(t)$$
 for all $t\in \N$, where $\bar{\vec{x}}(t) = \vec{x}(t)-  (\gamma(t)/\sqrt{N}) \sign(\nabla f(\vec{x}(t)))$. 
 Thus, from Equation~\eqref{eq:ThConDimBounded}, we can conclude that the sequence $\vec{x}(t)$ is bounded.

  \textbf{Step 4:} We will prove $\lim_{t\rightarrow \infty} \dist(\vec{x}(t),\mathcal{X}^{\star})=0$ by contradiction.
 Suppose that there exists $\epsilon>0$ and a subsequence $\vec{x}(t_k)$ such that $\dist(\vec{x}(t_k),\mathcal{X}^{\star})\geq \epsilon$ for all $k\in \N$.
  Then since $\vec{x}(t)$ is bounded, so we can without loss of generality restrict  $\vec{x}(t_k)$ to a convergent  subsequence to some point $\bar{\vec{x}}$, so $\lim_{k\rightarrow \infty} \vec{x}(t_k) = \bar{\vec{x}}$.
  Now since $f$ is continuous and $\lim_{t\rightarrow \infty} f(\vec{x}(t))=f^{\star}$ we can conclude that $f(\bar{\vec{x}})=f^{\star}$ and $\bar{\vec{x}}\in \mathcal{X}^{\star}$. 
 Then $\lim_{t\rightarrow \infty} \bar{\vec{x}}(t_k)=\vec{x} \in \mathcal{X}^{\star}$ contradicts that  $\dist(\vec{x}(t_k),\mathcal{X}^{\star})\geq \epsilon$ for all $k\in \N$.
\end{IEEEproof}


 \section{Proof of Lemma~\ref{lemma:binary_feedback_descent_lemma}} \label{App:lemma:binary_feedback_descent_lemma}
  \begin{IEEEproof}  
     By using that the gradients of $f$ are $L$-Lipschitz continuous, we can apply the descent lemma (see for example~\cite[eq.~(2.1.6)]{Book_Nesterov_2004} or~\cite[Proposition A.24]{nonlinear_bertsekas}). 
      The descent lemma states that for all $\gamma$ we have
      \begin{align}
         f(\vec{x}-\gamma \vec{d}) {\leq} & f(\vec{x}) - \langle \nabla f(\vec{x}), \vec{d}  \rangle \gamma 
                   {+} \frac{L}{2} ||\vec{d}||^2 \gamma^2,    \\
                           {=}&  f(\vec{x}) +\left(  \frac{L}{2}\gamma -\langle \nabla f(\vec{x}), \vec{d}(t)  \rangle    \right) \gamma,  \label{eq:prop1eq2}  \\
           \leq & f(\vec{x}) +\left(  \frac{L}{2}\gamma -\cos(\theta) \epsilon    \right) \gamma \label{eq:prop1eq3} \\
           =  & f(\vec{x}) - \delta(\epsilon,\gamma, \theta)
      \end{align}
     where Equation~\eqref{eq:prop1eq2} comes from that $||\vec{d}||{=}1$, Equation~\eqref{eq:prop1eq3} comes from that $\ang(  \nabla f( \vec{x}),\vec{d} ){\leq} \theta$, $||\nabla f(\vec{p})||{\geq} \epsilon$, since $\vec{x}\notin \mathcal{X}(\epsilon)$, and
   $  \langle \nabla f(\vec{x}), \vec{d}  \rangle  {=}   ||\nabla f(\vec{x})|| \cos(  \ang(\vec{d}, \nabla f(\vec{x})) )$. 
 \end{IEEEproof}

 \section{Proof of Lemma~\ref{prop:fixed-stepsize-type-2}} \label{App:prop:fixed-stepsize-type-2}
     \begin{IEEEproof}
   a)   The result can be proved using \textbf{Steps 1}, \textbf{2},  and \textbf{3} used to prove Lemma~\ref{prop:fixed-stepsize-type-2-Proj}, 
   using $||\nabla f(\cdot)||$, $\mathcal{X}(\epsilon)$, $F(\epsilon)$, and Lemma~\ref{lemma:binary_feedback_descent_lemma} in place of $L_{\alpha}(\cdot)$, $\bar{\mathcal{X}}_{\alpha}(\alpha)$, $\bar{F}_{\alpha}(\epsilon)$ and Lemma~\ref{lemma:binary_feedback_descent_lemma-Pojected}, respectively.

        b)
        For any $\vec{x} \in \mathcal{X}(\epsilon)$~\cite[eq.~(2.1.19) in Theorem 2.1.10]{Book_Nesterov_2004}
      \begin{align*}
        f(\vec{x}) \leq f^{\star}+ \frac{1}{2 \mu} ||\nabla f(\vec{x} )||^2 \leq  f^{\star}+ \frac{\epsilon^2}{2 \mu}, 
      \end{align*} 
     where we have used that $\nabla f(\vec{x}^{\star})=\vec{0}$ for all $\vec{x}^{\star}\in \mathcal{X}^{\star}$.    
    \end{IEEEproof}

\section{Proof of Theorem~\ref{prop:dim_step_size}} \label{App:prop:dim_step_size}
 \begin{IEEEproof} 
 The results can be proved using \textbf{Steps 1}, \textbf{2}, \textbf{3}, and \textbf{4} used to prove Theorem~\ref{prop:dim_step_size-Project}.
 The main difference is that here \textbf{Step 1} is to prove that $\liminf_{t\rightarrow \infty } ||\nabla f(\vec{x}(t))  ||= 0$, instead of $\liminf_{t\rightarrow \infty } L_{\alpha}(\vec{x}(t))= 0$ as in the proof of Theorem~\ref{prop:dim_step_size-Project}.
 Moreover, here we use  $||\nabla f(\cdot)||$, $\mathcal{X}(\epsilon)$, $F(\epsilon)$,  Lemma~\ref{lemma:binary_feedback_descent_lemma} and Lemma~\ref{prop:fixed-stepsize-type-2}-a) in place of $L_{\alpha}(\cdot)$, $\bar{\mathcal{X}}_{\alpha}(\alpha)$, $\bar{F}_{\alpha}(\epsilon)$, Lemma~\ref{lemma:binary_feedback_descent_lemma-Pojected}, and  Lemma~\ref{prop:fixed-stepsize-type-2-Proj}, respectively.
  \end{IEEEproof}

\section{Additional Lemmas} \label{App:AddLemmas} 
  \begin{lemma} \label{Lemma:Example1-longproofinappendix}
   Consider $\mathcal{D}$ defined in Equation~\eqref{eq:D:Examples:theta-cover-1} in Example~\ref{Examples:theta-cover-1} of Section~\ref{subsec:FLPQ}. $\mathcal{D}$ is a $\theta$-cover with the $\theta$ in Equation~\eqref{eq:inExample1_theta}. 
   \end{lemma} 
 \begin{IEEEproof}
  We show that for $\theta$ defined in Equation~\eqref{eq:inExample1_theta} it holds for any $\vec{x}\in \mathcal{S}^{N-1}$ that there exists $\vec{d}\in \mathcal{D}_1$ such that Equation~\eqref{defin:min_angle} holds.

  First consider the case where $\vec{x}_j {\geq} \cos(\theta)$ for some component $j$.
  Then  for $\vec{e}_j {\in} \mathcal{D}_1$ we get
  $\cos(\ang(\vec{x},\vec{e}_j)) = \langle \vec{x},\vec{e}_j \rangle=\vec{x}_j \geq \cos(\theta)$.
  Therefore, we finalize the proof by showing that if $\vec{x}\in \mathcal{S}^{N-1}$ and $\vec{x}_i \leq \cos(\theta)$ for $i=1,\ldots, N$ then 
  $$\cos\left(\ang\left(\vec{x},-\frac{1}{\sqrt{N}} \vec{1}\right)\right) = \frac{-1}{\sqrt{N}} \sum_{i=1}^N \vec{x}_i\geq \cos(\theta),$$
 Without loss of generality, let the components of $\vec{x}$ be ordered so that 
  $\vec{x}_i\geq 0$ if $i=1,\ldots, K$  and $\vec{x}_i< 0$  if   $i=K+1,\ldots, N$, 
  where $K$ is the number of positive components of $\vec{x}$.  
  Then 
  \begin{align}
     \frac{-1}{\sqrt{N}} \sum_{i=1}^N  \vec{x}_i \geq&  -\frac{1}{\sqrt{N}}\left(\sum_{i=1}^K\vec{x}_i - \sqrt{1- \sum_{i=1}^K\vec{x}_i^2  } \right) \label{eq:inlemmaEx1-1-a}\\
     \geq&   {-}\frac{1}{\sqrt{N}}\left(K \cos(\theta) - \sqrt{1- K \cos(\theta)^2  } \right),  \label{eq:inlemmaEx1-1-c}
  \end{align}
  where Equation~\eqref{eq:inlemmaEx1-1-a} comes by using that $\sum_{i=1}^N \vec{x}_i^2=1$ and the inequality between the $1$ and $2$ norm, i.e., 
  \begin{align} \notag
        \sum_{i=K+1}^N |\vec{x}_i| \geq \sqrt{ \sum_{i=K+1}^N \vec{x}_i^2} = \sqrt{ 1- \sum_{i=1}^K \vec{x}_i^2},
  \end{align} 
 and Equation~\eqref{eq:inlemmaEx1-1-c}  comes by noting that~\eqref{eq:inlemmaEx1-1-a} is decreasing and that $\vec{x}_i\leq \cos(\theta)$ for all $i$.
  Now, by inserting our choice of $\cos(\theta)$ from Equation~\eqref{eq:inExample1_theta} in Equation~\eqref{eq:inlemmaEx1-1-c} we get
  \begin{align}
         \frac{{-}1}{\sqrt{N}} \sum_{i=1}^N \vec{x}_i    
         {\geq} & {-} \frac{K{-}\sqrt{N^2{+}2\sqrt{N}(N{-}1){-}K}}{\sqrt{N}  \sqrt{N^2+2\sqrt{N}(N-1)}}  \label{eq:inlemmaEx1-2-a}\\
         \geq & {-}  \frac{N{-}1{-}\sqrt{N^2{+}2\sqrt{N}(N{-}1){-}(N{-}1)}}{\sqrt{N} \sqrt{N^2+2\sqrt{N}(N-1)}}   \label{eq:inlemmaEx1-2-b} \\
         = &   \frac{\sqrt{N}}{\sqrt{N} \sqrt{N^2+2\sqrt{N}(N-1)}}  ~~=~~\cos(\theta) \label{eq:inlemmaEx1-2-c} 
  \end{align}
  where the Equation~\eqref{eq:inlemmaEx1-2-b} comes from the fact that Equation~\eqref{eq:inlemmaEx1-2-a} is decreasing in $K$ and $K\leq N-1$, and the Equation~\eqref{eq:inlemmaEx1-2-c} comes by using that 
 $ N^2{+}2\sqrt{N}(N{-}1){-}(N{-}1) = ((N{-}1){+}\sqrt{N})^2.$
  \end{IEEEproof}
 
\begin{lemma}\label{Lemma:projections}
 For all $\beta \in [0,1]$, $z\in \R$ and $x,\alpha_1,\alpha_2 \in \R_+$  with $\alpha_1 \leq \alpha_2$ following holds 
 \begin{align}
  \beta |x-\lceil x-z\rceil^+ | &\leq |x- \lceil x- \beta z\rceil^+|, \label{eq:LemPro-1}\\
   |x-\lceil x-\alpha_1 z\rceil^+ | &\leq |x- \lceil x- \alpha_2 z\rceil^+|, \label{eq:LemPro-2} \\
   0 &\leq z(x-\lceil x-\alpha_1 z\rceil^+ ). \label{eq:LemPro-3} 
 \end{align}
\end{lemma}
\begin{IEEEproof}
   We first prove~\eqref{eq:LemPro-1}.
   Direct inspection shows that 
  \begin{align}
    \phi_1(x,z,\beta) :=
     \beta |x-\lceil x-z\rceil^+ |  &=\begin{cases} \beta |z| & \text{ if } x\geq z \\ \beta x & \text{ if } x \leq  z \end{cases} \label{eq:ClosedForm-1} \\
    \phi_2(x,z,\beta) :=
     |x-\lceil x- \beta z\rceil^+ | &=\begin{cases} \beta |z| & \text{ if } x\geq \beta z \\ x & \text{ if } x \leq \beta z. \end{cases} \label{eq:ClosedForm-2}
  \end{align}
  Therefore, for  $z\in \R_+$ we have $\phi_1(x,z,\beta)= \beta |z| = \phi_2(x,z,\beta)$  if  $x \in [z,\infty)$, 
  $\phi_1(x,z,\beta) = \beta x \leq \beta |z| =\phi_2(x,z,\beta)$  if $x \in [\beta z,z]$, 
   $\phi_1(x,z,\beta) = \beta x \leq x =\phi_2(x,z,\beta)$ if $x \in [0,\beta z]$.
  So $ \phi_1(x,z,\beta)\leq  \phi_2(x,z,\beta)$ for all $x,z\in \R_+$ and $\beta\in [0,1]$ which yields~\eqref{eq:LemPro-1}. 
  
  Equation~\eqref{eq:LemPro-2}  follows directly from using~\eqref{eq:ClosedForm-2}. 
%
     To prove~\eqref{eq:LemPro-3}, we use the fact that
     $\sign(z) \lceil x-\alpha_1 z\rceil^+ \leq \sign(z) x$
   or by rearranging 
    $0 \leq \sign(z) (x - \lceil x-\alpha_1 z\rceil^+).$
    By multiplying $|z|$ on both sides we obtain~\eqref{eq:LemPro-3}.
\end{IEEEproof}

 \begin{lemma}  \label{lemma:bounded_Xeps}
   Suppose $\mathcal{X}^{\star}$ is bounded. 
   Then:
    (i) There exists $\kappa{>}0$ such that $\mathcal{X}(\epsilon)$ defined in Equation~\eqref{eq:inProP-FSS-SC-Type-1-MCXep} is bounded for all $\epsilon {<}\kappa$. 
    (ii) If $||\nabla f(\vec{x})||{\leq} B$ for all $\vec{x}{\in} \R_+^N$, then there exists $\kappa{>}0$ such that $\bar{\mathcal{X}}_{\alpha}(\epsilon)$ in Equation~\eqref{eq:inProP-FSS-SC-Type-1-MCXep-Proj} is bounded for all $\epsilon {<}\kappa$.
 \end{lemma}
 \begin{IEEEproof} (i)  
   Take any $\vec{x}^{\star}\in \mathcal{X}^{\star}$ and choose $R>0$ so that $\mathcal{X}^{\star}\subseteq \mathcal{B}^N(\vec{x}^{\star},R)$. 
   Take $\kappa_1>0$ given by
   \begin{align} \label{inLemma:kappa}
       \kappa_1 = \frac{1}{L} ~\underset{\vec{x} \in \mathcal{S}^{N-1}(\vec{x}^{\star},R)}{\mbox{minimize}}~ ||\nabla f(\vec{x}) ||^2.
   \end{align}
   Note that such a $\kappa_1$ exists since $ \mathcal{S}^{N-1}(\vec{x}^{\star},R)$ is compact and $\kappa_1>0$ since   $ \mathcal{S}^{N-1}(\vec{x}^{\star},R) \cap \mathcal{X}^{\star}$ is empty.
   Moreover, using~\cite[(2.1.8) in Theorem 2.1.5]{Book_Nesterov_2004}, that $\nabla f(\vec{x}^{\star})=\vec{0}$, and~\eqref{inLemma:kappa}, we have for all $\vec{x}\in  \mathcal{S}^{N-1}(\vec{x}^{\star},R)$ that
    \begin{align} \label{inLemma:kappaineq}
     \langle \nabla f(\vec{x}), \vec{x}-\vec{x}^{\star}\rangle \geq (1/L) ||\nabla f(\vec{x})||^2 \geq \kappa_1.
   \end{align}

   We now show that for all $\vec{x}\in \R^N\setminus \mathcal{B}^N(\vec{x}^{\star},R)$ we have $||\nabla f(\vec{x})||\geq \kappa$, where $\kappa=\kappa_1 /R$.
   Take some $\vec{x}\in \R^N\setminus \mathcal{B}^N(\vec{x}^{\star},R)$ and let $\bec{x}$ denote the unique point in the intersection of the line segment $[\vec{x}^{\star},\vec{x}]$ and  $\mathcal{S}^{N-1}(\vec{x}^{\star},R)$, such a  $\bec{x}$ exists because $\vec{x}\in \R^N\setminus \mathcal{B}^N(\vec{x}^{\star},R)$ and $\vec{x}^{\star} \in \mathcal{B}^N(\vec{x}^{\star},R)$.
   Consider now the function $G:[0,\infty) \rightarrow \R^N$ with
   \begin{align} \label{eq:defin_G}
      G(\tau) = \nabla f(\vec{x}^{\star}+ \tau (\bec{x}-\vec{x}^{\star})).
   \end{align}
   Clearly, $G(0)=\nabla f(\vec{x}^{\star})=\vec{0}$, $G(1)=\nabla f(\bec{x})$ and there exists $\hat{\tau}\geq 1$ such that $G(\hat{\tau})=\nabla f(\vec{x})$. 
   By using that gradients of convex functions are monotone, i.e., for all $\vec{x}_1,\vec{x}_2\in \R^N$ it holds that
   $\langle \nabla f(\vec{x}_1)-\nabla f(\vec{x}_2), \vec{x}_1-\vec{x}_2\rangle \geq 0$,
   we conclude that  for  $\tau_1, \tau_2\in \R_+$ with $\tau_1\geq \tau_2 $ it holds that $\langle G(\tau_1)-G(\tau_2),(\tau_1-\tau_2) (\bec{x}-\vec{x}^{\star}) \rangle \geq 0$.
   Rearranging  this,
   \begin{align}\label{inLemma:kappaineq2}
         \langle G(\tau_1), (\bec{x}-\vec{x}^{\star}) \rangle \geq  \langle G(\tau_2), (\bec{x}-\vec{x}^{\star}) \rangle,~~\text{ for } \tau_1\geq \tau_2. 
   \end{align}
    By  combining~\eqref{inLemma:kappaineq} and~\eqref{inLemma:kappaineq2} we get that
      \begin{align} \label{eq:lemma-kappa3}
         \langle G(\hat{\tau}), (\bec{x}-\vec{x}^{\star}) \rangle \geq  \langle G(1), (\bec{x}-\vec{x}^{\star}) \rangle \geq \kappa_1.  
   \end{align}
   Hence, by  the Cauchy-Schwarz inequality we have
       $|| \nabla f(\vec{x})|| R =   || G(\hat{\tau})|| R \geq \kappa_1$
   and by rearranging we get $|| \nabla f(\vec{x})||  \geq \kappa_1/R =\kappa$. 
   Since $|| \nabla f(\vec{x})||\geq \kappa$ holds for all $\vec{x}\in  \R^N\setminus \mathcal{B}^N(\vec{x}^{\star},R)$ we
   can conclude that  $\mathcal{X}(\epsilon)$ is bounded for $\epsilon <\kappa$.

 (ii) 
 We prove the result by contradiction. Suppose $\bar{\mathcal{X}}_{\alpha}(\kappa)$ is unbounded for all $\kappa>0$. 
    Then there exists a sequence $\vec{x}^k\in \R_+^N$  such that $\lim_{k\rightarrow \infty}||\vec{x}^k||=\infty$ and $\lim_{k\rightarrow \infty}L_{\alpha}(\vec{x}^k)=0$. 
    We prove the contraction in the following steps: 

  \textbf{Step 1:} We will show that there exists $\bar{\kappa}>0$ and $R$ such that $ || \nabla f(\vec{x})|| \geq \bar{\kappa}$ holds for all $\vec{x}\in \R_+^N$  and $||\vec{x}-\vec{x}^{\star}||\geq R$. 
  If there exists $\vec{x}^{\star}\in \mathcal{X}^{\star}$ such that $||\nabla f(\vec{x}^{\star})||=0$, then the result follows from part (i).
   Therefore, without loss of generality, suppose we can take $\vec{x}^{\star}\in \mathcal{X}^{\star}$ with $||\nabla f(\vec{x}^{\star})||>0$. 
 Then the set   $\mathcal{J}:=\{j =1,\ldots, N, |\nabla_j f(\vec{x}^{\star}) \neq 0 \}$ is nonempty. 
  We also have, using the KKT conditions~\cite[Section 5.9.2]{convex_boyd}, that  $\vec{x}\in \mathcal{X}^{\star}$ if and only if the following three conditions hold (A) $\vec{x}\in \R_+^N$, (B) $\nabla_i f(\vec{x})\geq 0$ for $i=1,\ldots, N$, and (C) $\nabla_i f(\vec{x})\vec{x}_i=0$ for $i=1,\ldots, N$.~\footnote{The Lagrangian multiplier associated with $\vec{x}\in \mathcal{X}^{\star}$ is $\boldsymbol \lambda{=} \nabla f(\vec{x})$.}

  We first show that $\langle \nabla f(\vec{x}),\vec{x}-\vec{x}^{\star}\rangle>0$ for all $\vec{x} \in \R_{+}^N\setminus \mathcal{X}^{\star}$. 
  Consider first the case when $\vec{x} \in \R_{+}^N\setminus \mathcal{X}^{\star}$ and $\vec{x}_j>0$ for some $j\in \mathcal{J}$.
  Then we have  
  \begin{align*}  
   \langle \nabla f(\vec{x}), \vec{x}-\vec{x}^{\star}\rangle\geq  \langle \nabla f(\vec{x}^{\star}), \vec{x}-\vec{x}^{\star}\rangle 
     \geq  \sum_{i=1}^N \nabla_i f(\vec{x}^{\star}) \vec{x}_i  >0,
  \end{align*}
  where the first inequality comes by the monotonicity of $\nabla f$, the second  inequality comes by the optimality condition (C), 
   and the final inequality comes by the optimality condition (B), the fact that $\nabla_i f(\vec{x}^{\star})>0$ for all $j\in \mathcal{J}$, and that $\vec{x}_j>0$ for some $j\in \mathcal{J}$. 
   Consider next the case when $\vec{x} \in \R_{+}^N\setminus \mathcal{X}^{\star}$ and $\vec{x}_j=0$ for all $j\in \mathcal{J}$. 
   Then $\nabla_i f(\vec{x})\neq \nabla_i f(\vec{x}^{\star})$ for some $i$, because otherwise the optimality conditions (A), (B), and (C) hold for $\vec{x}$ so $\vec{x}\in \mathcal{X}^{\star}$.
   In particular, $||\nabla f(\vec{x})-\nabla f(\vec{x}^{\star})||>0$.  
   Therefore, we have~\cite[eq.~(2.1.8)]{Book_Nesterov_2004} 
   \begin{align*}
     \langle \nabla f(\vec{x}), \vec{x}{-}\vec{x}^{\star}\rangle{\geq}  \langle \nabla f(\vec{x}^{\star}), \vec{x}{-}\vec{x}^{\star}\rangle 
    {+} \frac{1}{L} ||\nabla f(\vec{x}){-}\nabla f(\vec{x}^{\star}) ||^2{>}0,
   \end{align*}
   where the final inequality comes by that $ \langle \nabla f(\vec{x}^{\star}), \vec{x}-\vec{x}^{\star}\rangle\geq 0$ for all $\vec{x}\in \R_+^N$ and that $||\nabla f(\vec{x})-\nabla f(\vec{x}^{\star})||>0$.

  Now take $R>0$ such that $\mathcal{X}^{\star}\subseteq \mathcal{B}^N(\vec{x}^{\star},R)$.  
  Then since $\mathcal{S}^{N-1}(\vec{x}^{\star},R)\cap \R_+^N$ is compact, there exists 
   $\kappa_1= \min_{\vec{x}\in \mathcal{S}^{N-1}(\vec{x}^{\star},R)\cap \R_+^N}   \langle \nabla f(\vec{x}), \vec{x}-\vec{x}^{\star}\rangle>0$.
 We can now follow  same arguments as in the proof of part (i) to show that $||\nabla f(\vec{x})||\geq \bar{\kappa}$ where $\bar{\kappa}=\kappa_1/R$.

 \textbf{Step 2:} We will show that  the following inequality holds for all $\vec{x}\in \mathcal{R}_+^N\setminus \mathcal{B}^N(\vec{x}^{\star},R)$,
 \begin{align} \label{eq:lastlemma-cosineq}
      \cos(\ang(\nabla f(\vec{x}),\vec{x}-\vec{x}^{\star})) \geq \frac{\bar{\kappa}}{B}, 
 \end{align}
 where $R$ and $\bar{\kappa}$ are defined as in \textbf{Step 1}.
Take some $\vec{x}\in \R_+^N\setminus \mathcal{B}^N(\vec{x}^{\star},R)$. 
 Similarly as in part (i), let $\bar{\vec{x}}$ denote the unique point in the intersection of the line segment $[\vec{x}^{\star},\vec{x}]$ and  $\mathcal{S}^{N-1}(\vec{x}^{\star},R)$. 
 Moreover, take $\hat{\tau}>1$ such that $\vec{x}=\vec{x}^{\star}+\hat{\tau}(\bar{\vec{x}}-\vec{x}^{\star})$ and
 define $G:[0,\infty)\rightarrow \R^N$ as in Equation~\eqref{eq:defin_G}.
 Then by rearranging~\eqref{eq:lemma-kappa3} and multiplying both sides with $1/R$  
 \begin{align} \notag
      \cos(\ang(\nabla f(\vec{x}),\vec{x}-\vec{x}^{\star})) \geq \frac{\kappa_1}{R} \frac{1}{||G(\hat{\tau})||}\geq \frac{\bar{\kappa}}{B},
 \end{align}
 where $\kappa_1$ and $\bar{\kappa}$ are defined as in part (i) and \textbf{Step 1}.

  \textbf{Step 3:} We will show that the subsequence $\vec{x}^k$ can be restricted so that (a) $\lim_{k\rightarrow \infty}  \nabla f(\vec{x}^k) = \vec{f}$ for some $\vec{f}\in\R^N$ and (b) for each component $i=1,\ldots,N$ either $\lim_{k\rightarrow \infty} \vec{x}_i^k=0$ or $\vec{x}_i\geq W$, for some $W>0$.
 We first show (a). Since  $\nabla f$ is bounded by $B$,
  the sequence $\nabla f(\vec{x}^k)$ is bounded. 
  Therefore, we can restrict the sequence $\vec{x}^k$ so that  $\nabla f(\vec{x}^k)$ is a convergent subsequence with $\lim_{k\rightarrow \infty}  \nabla f(\vec{x}^k) = \vec{f}$. 
  To show (b),
 for each component $i=1,\ldots, N$ we restrict the sequence $\vec{x}^k$ so that $\vec{x}_i^k\geq W_i$ if $\vec{x}_i^k$ does not converge to $0$ 
 and taking $W=\max W_i$. 

     \textbf{Step 4:}   
       We will prove that $\vec{f}_i=0$ for $i\notin\mathcal{I}:=\{i=1,\ldots,N \big| \lim_{k\rightarrow \infty} \vec{x}_i^k=0\}.$. We prove the result by contradiction. 
       Without loss of generality, suppose $\vec{f}_i>0$ for some  $i\notin\mathcal{I}$,  the case when $\vec{f}_i<0$ follows same arguments. 
       Then there exists $K\in \N$ such that $\nabla_i f(\vec{x}^k)\geq \eta_0:= \vec{f}_i/2>0$ for all $k\geq K$.
       This, together with that   $\vec{x}_i^k\geq W$ implies that
           $|\vec{x}_i^k{-}\lceil \vec{x}_i^k{-}\nabla_i f(\vec{x}^k)\rceil_+| {\geq} \min \{ \eta_0,W\}=: \eta{>}0$.
       Therefore, $L_{\alpha}(\vec{x}^k)\geq \eta$ for all $k\geq K$, contradicting that $\lim_{t\rightarrow \infty} L_{\alpha}(\vec{x}^k)=0$.
       
     \textbf{Step 5:} We will prove contradiction when $\mathcal{I}$ 
     is empty.
     From \textbf{Step 4}, we have $\vec{f}=\vec{0}$.    
      However, we also have from \textbf{Step 1}  that $||\nabla f(\vec{x})||\geq \bar{\kappa}$ for all $\vec{x}\in  \R_+^N \setminus \mathcal{B}^N(\vec{x}^{\star},R)$. 
        Since  $\mathcal{B}^N(\vec{x}^{\star},R)$ is bounded and $\lim_{t\rightarrow \infty}||\vec{x}(t)||=\infty$, 
        we have that $||\vec{f}||\geq \bar{\kappa}>0$, which contradicts that $\vec{f}=\vec{0}$.

     \textbf{Step 6:} We will prove contradiction when $\mathcal{I}$ is nonempty. 
     Consider the sequence
      $ \vec{z}^k=(\vec{x}^k{-}\vec{x}^{\star})/(||\vec{x}^k{-}\vec{x}^{\star}||)$. 
     Since $\vec{z}^k$ is bounded, we can restrict the subsequence $\vec{x}^k$ so that $\vec{z}^k$ has a convergent subsequence, with the limit $\vec{z}$. 
    We have $\vec{z}_i{=}0$ for $i\in \mathcal{I}$ and $||\vec{z}||{=}1$, since $||\cdot||$ is continuous function and $\vec{z}^k$ convergent sequence.
 Therefore, as both $\nabla f(\vec{x}^k)$ and $\vec{z}^k$ are convergent sequences and the inner product $\langle \cdot, \cdot \rangle$ is continuous function, the sequence  $\langle \nabla f(\vec{x}^k), \vec{z}^k \rangle$ is convergent and has the limit $ \langle \vec{f}, \vec{z} \rangle = 0$. 
 However, for all $\vec{x}^k\in \mathcal{R}_+^N\setminus \mathcal{B}^N(\vec{x}^{\star},R)$
 \begin{align}
  \langle \nabla f(\vec{x}^k), \vec{z}^k \rangle {=} ||\nabla f(\vec{x}^k)|| \cos(\ang(\nabla f(\vec{x}^k), \vec{z}^k))
      {\geq} \frac{\bar{\kappa}^2}{B}, \notag
 \end{align}
 where the inequality comes by Equation~\eqref{eq:lastlemma-cosineq} in \textbf{Step 2}.
 \end{IEEEproof}

\bibliography{refs}
\bibliographystyle{IEEEtran}

\begin{IEEEbiography}[{\includegraphics[width=1in,height=1.25in,clip,keepaspectratio]{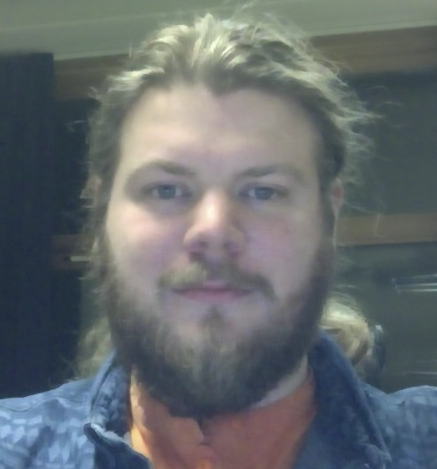}}]{Sindri Magn\'usson}
received the B.Sc. degree in Mathematics from University of Iceland, Reykjav\'ik Iceland, in 2011, the Masters degree in Mathematics from KTH Royal Institute of Technology, Stockholm Sweden, in 2013, and the PhD in Electrical Engineering from the same institution, in 2017. He is currently a postdoctoral researcher at the Department of Electrical Engineering at KTH Royal Institute of Technology, Stockholm, Sweden. He was a visiting researcher at Harvard University, Cambridge, MA, for 9 months in 2015 and 2016. His research interests include distributed optimization, both theory and applications.
\end{IEEEbiography}

  \begin{IEEEbiography}[{\includegraphics[width=1in,height=1.25in,clip,keepaspectratio]{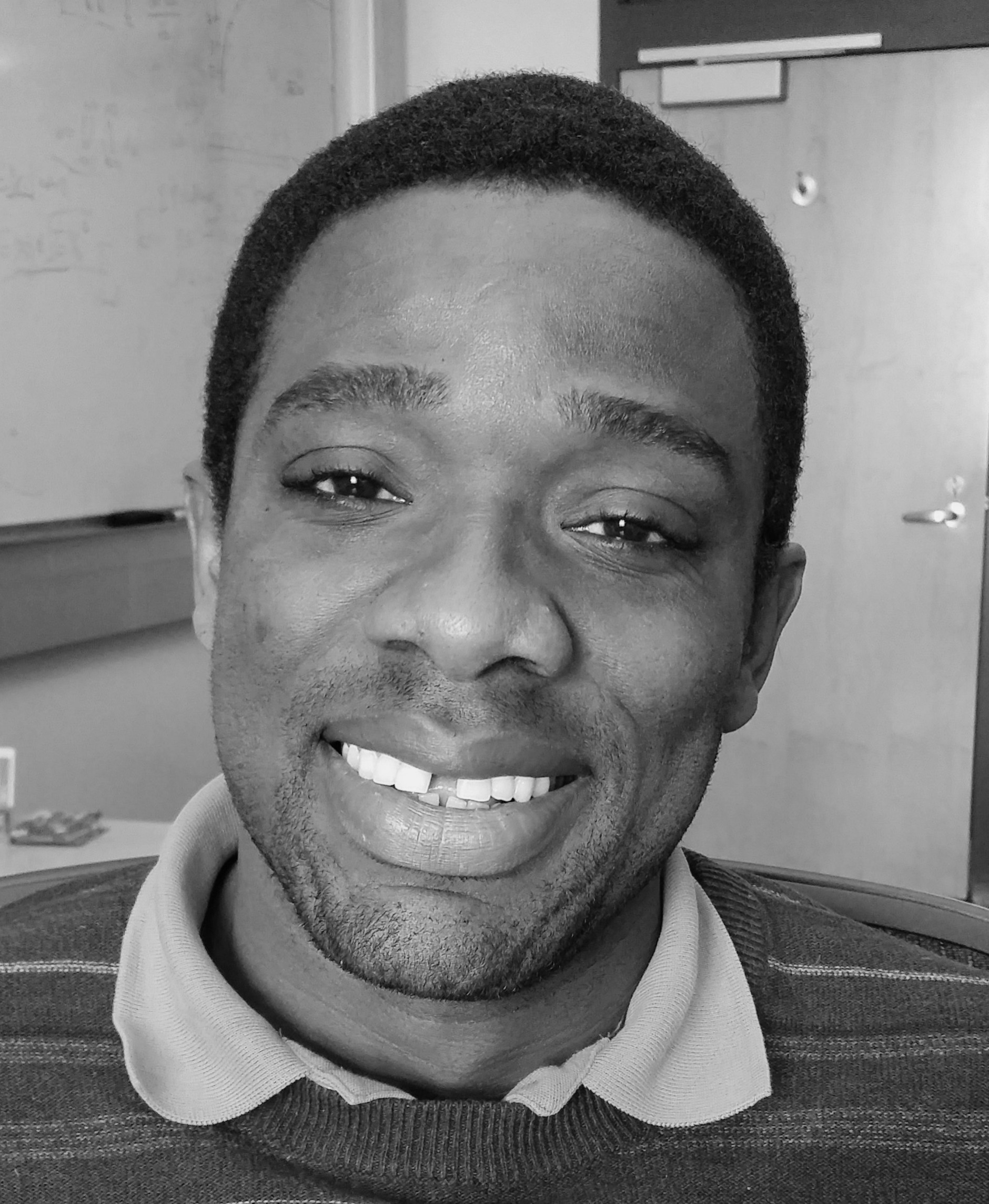}}]
{Chinwendu Enyioha} 
 received the B.Sc. degree in Mathematics from Gardner-Webb University (GWU), Boiling Springs, NC,  and the PhD degree in Electrical and Systems Engineering from the University of Pennsylvania, Philadelphia, PA, in 2014. He is currently a Postdoctoral Research Fellow in the School of Engineering and Applied Sciences at Harvard.  Prior to arriving Harvard, he was a Postdoctoral Researcher in the GRASP Lab at the University of Pennsylvania. Dr. Enyioha is a Fellow of the Ford Foundation, was named a William Fontaine Scholar at the University of Pennsylvania and has received the Mathematical Association of America Patterson award. His research lies in the area of design, optimization and control of distributed networked systems, with applications to power systems and Cyber-Physical networks.
\end{IEEEbiography}

 \begin{IEEEbiography}[{\includegraphics[width=1in,height=1.25in,clip,keepaspectratio]{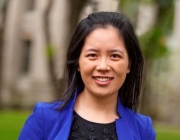}}]{Na Li} 
  received her B.S. degree in mathematics and applied mathematics from Zhejiang University in China and her PhD degree in Control and Dynamical systems from the California Institute of Technology in 2013. She is an Assistant Professor in the School of Engineering and Applied Sciences in Harvard University. She was a postdoctoral associate of the Laboratory for Information and Decision Systems at Massachusetts Institute of Technology. She was a Best Student Paper Award finalist in the 2011 IEEE Conference on Decision and Control.    
 She received NSF CAREER Award in 2016 and Air Force Young Investigator Award in 2017.
  Her research lies in the design, analysis, optimization and control of distributed network systems, with particular applications to power networks and systems biology/physiology.
\end{IEEEbiography}

 \begin{IEEEbiography}[{\includegraphics[width=1in,height=1.25in,clip,keepaspectratio]{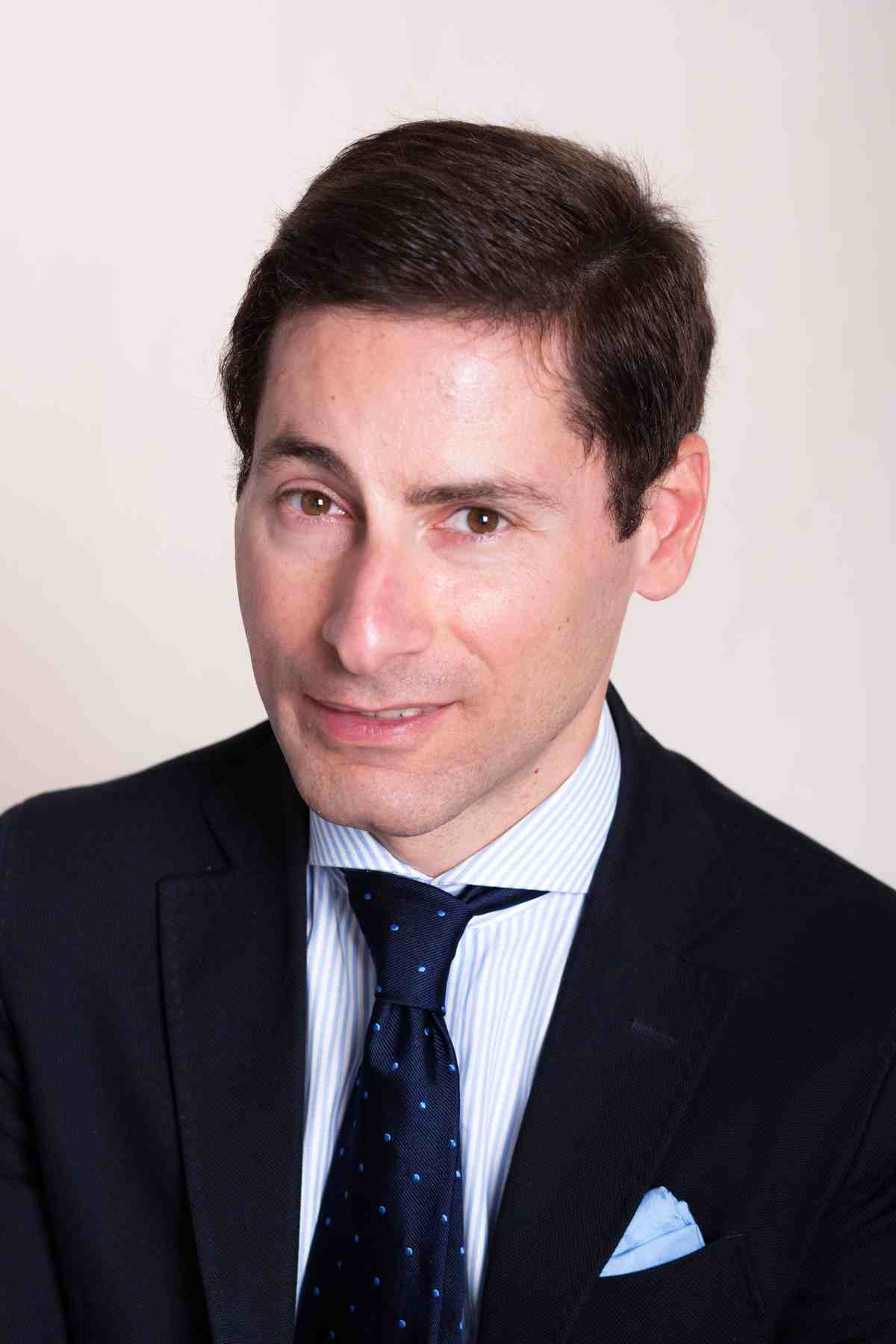}}]{Carlo Fischione} 
 is currently a tenured Associate Professor at KTH Royal Institute of Technology, Electrical Engineering, Stockholm, Sweden. He received the Ph.D. degree in Electrical and Information Engineering (3/3 years) in May 2005 from University of L’Aquila, Italy, and the Laurea degree in Electronic Engineering (Laurea, Summa cum Laude, 5/5 years) in April 2001 from the same University. He has held research positions at Massachusetts Institute of Technology, Cambridge, MA (2015, Visiting Professor); Harvard University, Cambridge, MA (2015, Associate); University of California at Berkeley, CA (2004-2005, Visiting Scholar, and 2007-2008, Research Associate). His research interests include optimization with applications to wireless sensor networks, networked control systems, wireless networks, security and privacy. He received or co-received a number of awards, including the best paper award from the IEEE Transactions on Industrial Informatics (2007).  He is Member of IEEE (the Institute of Electrical and Electronic Engineers), and Ordinary Member of DASP (the academy of history Deputazione Abruzzese di Storia Patria).
\end{IEEEbiography} 

 \begin{IEEEbiography}[{\includegraphics[width=1in,height=1.25in,clip,keepaspectratio]{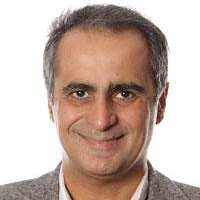}}]
{Vahid Tarokh}
 received the PhD in electrical engineering from the University of Waterloo, Ontario, Canada. He then worked at AT\&T Labs-Research and AT\&T wireless services until August 2000, where he was the head of the Department of Wireless Communications and Signal Processing.

In September 2000, he joined the Department of Electrical Engineering and
Computer Sciences (EECS) at MIT as an associate professor. In June 2002,
he joined Harvard University, where he is a Professor of Applied Mathematics.
He has received a Guggenheim Fellowship in Applied Mathematics and holds 4 honorary degrees.
\end{IEEEbiography} 



\end{document}